\documentclass{siamltex}
\usepackage{amsmath}
\usepackage{amsfonts}
\usepackage{amssymb}
\usepackage{graphicx}
\usepackage{color}
\usepackage{bm}
\usepackage{verbatim}
\usepackage{hyperref}
\usepackage{bmpsize}
\usepackage{algorithm}
\usepackage[noend]{algpseudocode}
\allowdisplaybreaks
\usepackage{amsxtra,comment,graphicx,psfrag}
\usepackage{epstopdf}
\usepackage{subfigure}
\usepackage{float}
\usepackage{cases}

\newtheorem{Aalgorithm}[theorem]{Algorithm}

\newcommand{\be}{\begin{equation}}
\newcommand{\ee}{\end{equation}}
\newcommand{\bea}{\begin{eqnarray}}
\newcommand{\eea}{\end{eqnarray}}
\newcommand{\beas}{\begin{eqnarray*}}
\newcommand{\eeas}{\end{eqnarray*}}

\newcommand{\vertiii}[1]{{\left\vert\kern-0.25ex\left\vert\kern-0.25ex\left\vert #1 
    \right\vert\kern-0.25ex\right\vert\kern-0.25ex\right\vert}}

    \makeatletter
\def\BState{\State\hskip-\ALG@thistlm}
\makeatother

\begin{document}
\title{Three linear, unconditionally stable, second order decoupling methods for the Allen--Cahn--Navier--Stokes phase field model}

\author{
Ruonan Cao\thanks{Department of Mathematics, Shantou University, Guangdong, 515063 China, {\tt 19rncao@stu.edu.cn}. }.
\and Nan Jiang\thanks{Department of Mathematics, University of Florida, Gainesville, FL 32611, {\tt jiangn@ufl.edu}.}
\and Huanhuan Yang \thanks{Corresponding author. Department of Mathematics, Shantou University, Guangdong, 515063 China, {\tt huan2yang@stu.edu.cn}. }
}
\maketitle

\begin{abstract}
Hydrodynamics coupled phase field models have intricate difficulties to solve numerically as they feature high nonlinearity and great complexity in coupling. In this paper, we propose three second order, linear, unconditionally stable decoupling methods based on the Crank--Nicolson leap-frog time discretization for solving the Allen--Cahn--Navier--Stokes (ACNS) phase field model of two-phase incompressible flows. The ACNS system is decoupled via the artificial compression method and a splitting approach by introducing an exponential scalar auxiliary variable.
We prove all three algorithms are unconditionally long time stable. Numerical examples are provided to verify the convergence rate, unconditional stability, and computational efficiency.

\end{abstract}


\section{Introduction}
Phase field or diffuse interface models are widely used to study the interfacial dynamics in many scientific and engineering applications \cite{Anderson98, Edwards91, Feng06, Han17}. This topic has received an increasing amount of attention from the research community in mathematical and numerical analysis in the past decade.
In particular, the ACNS phase field model is popularly used to describe the motion of a mixture of two incompressible fluids \cite{Liu, Pyue}.  It features a nonlinear system consisting of the incompressible Navier--Stokes (NS) equations coupled with the Allen--Cahn (AC) equations.
In the model, a continuous phase field function $\phi$ is introduced to label different fluid components while their sharp interface is implicitly tracked by a thin but smooth transition layer, i.e. the diffuse interface. 
By studying the evolution of the phase field function, one can avoid explicit interface tracking and perform simulations on a fixed mesh grid, rendering a convenient numerical approach to simulate various interfacial problems.
The dynamics of the phase field variable is described by the Allen--Cahn equation, obtained by the gradient flow method, namely, minimizing the total free energy in the space of $L^2$.

There exist many effective numerical schemes for each component of the ACNS system, for instance, the projection method \cite{Chorin68, Rannacher91, GMS06, Guermond09} and the artificial compression method \cite{DLM17, CLM19, LM20,  GM19, DILMS20} for the incompressible Navier-Stokes equations; the convex splitting method \cite{ShenWang12, Wise13, HW15}, the Lagrange multiplier approach \cite{Guillen13}, the Invariant Energy Quadratization (IEQ) method 
\cite{YZW17}, and the Scalar Auxiliary Variable approach (SAV) \cite{Shen18, ZYang18} for the phase-field type models. 
For solving a hydrodynamics coupled phase field model, however, a simple combination of the methods from each component may not produce an accurate, efficient, and unconditionally long time stable scheme.
The main challenges lie in the high nonlinearity in the Allen--Cahn and Navier--Stokes equations, the coupling between the phase field variable $\phi$ and velocity $u$ through a phase induced stress term in the NS equations and a fluid induced transport term in the Allen--Cahn equations, and the coupling between the fluid velocity $u$ and pressure $p$ in the Navier--Stokes equations.

Due to the intricate complexity of the ACNS system, most of the schemes developed in the literature are either first-order time accurate \cite{Shen10, Shen15}, or nonlinear and coupled schemes \cite{Han17, Vaibhav18}, or fully decoupled but without provable energy stability analysis \cite{Dong12, He21}, see an overview also in \cite{Han17, xfYang20}. The first method to have all the desired characteristics (i.e., linear, unconditionally energy stable, fully-decoupled, and second-order accurate in time ) for solving the NS equations coupled with mass-conserved Allen--Cahn phase field model is based on the second-order backward differentiation formula (BDF2) for time stepping and the projection method for velocity and pressure decoupling \cite{xfYang20}. 
 Nevertheless, the projection method is only first-order accurate for the pressure due to the artificial boundary conditions applied to the pressure, and one needs to solve a Poisson equation in each time step to update the pressure.  
Developing other alternative efficient and unconditionally long time stable, higher order time stepping methods is still in great need.

While existing methods are based on either the Crank--Nicolson or the BDF2 time stepping, we focus on the Crank--Nicolson leap-frog time discretization (CNLF \cite{Layton12, Layton14, Jiang15a, Jiang15b}) to develop three efficient, linear, unconditionally long time stable decoupling schemes. The CNLF method is commonly used in the atmospheric and oceanic simulations for its high accuracy, but has been less theoretically studied in the literature, and used for other engineering applications. Herein we design three numerical schemes based on CNLF for the the phase field model and relevant applications.
In the first scheme, we adopt an idea of Lagrange multiplier of \cite{Guillen13} to linearize the Allen--Cahn equations, then combine it with the artificial compression technique for decoupling the velocity and pressure in the NS equations. 
Unlike the most frequently employed projection type methods which were first introduced by Chorin \cite{Chorin68} and Temam \cite{Temam69} in the late 1960s, the artificial compression methods which were also first studied in the sixties by Chorin \cite{Chorin67}, Temam \cite{Temam69a,Temam69}, Kuznetsov, Vladimirova and Yanenko \cite{KVY66}, are less studied in the literature and have only recently received increasing attention. 
The artificial compression methods relax the incompressibility constraint in the NS equations by adding a perturbation, e.g., $ \epsilon\partial_tp$, to the mass conservation equation which facilitates decoupling the computation of velocity and pressure in time marching schemes, see \cite{DLM17, CLM19, LM20,  GM19, DILMS20} for recent developments. 
It is worth noting the outstanding feature of the artificial compression methods is that the pressure can be updated directly without solving a Poisson equation which avoids the spurious oscillations in the boundary layer of pressure due to artificial boundary conditions required in a projection method. We will derive a linear and unconditionally stable scheme that is partially decoupled. Despite the fact that the computation of the phase field variable is still coupled with that of the velocity, we only need to solve a linear system in reduced size without the use of Picard/Newton iterations, which greatly reduces the computational cost.

The second idea is to incorporate the artificial compression technique with an SAV decoupling strategy for developing highly efficient, fully decoupled schemes. 
The SAV approach was first studied in \cite{SX18, Shen18} for gradient flows. It introduces a new scalar auxiliary variable that can be used to form a modified system of the underlying partial differential equation (PDE) system so that the nonlinear part can be canceled out in stability analysis, leading to unconditionally stable methods for solving nonlinear systems \cite{LS20, LYD19}. Following the new decoupling strategy studied in \cite{JY21b} for the Stokes-Darcy equations using the SAV idea, we find it is also possible to cancel out the coupling terms that usually lead to the time step constraints in a typical decoupling method for the ACNS model.  
This is achieved by introducing a scalar auxiliary variable to handle the lagged coupling terms in the ACNS model, namely the phase induced stress term in the NS equations and the fluid induced transport term in the Allen--Cahn equations. A modified PDE system which is equivalent to the original ACNS model is then formulated, and an efficient, fully decoupled discretization method can be derived and proved to be long time stable without any time step constraints. The SAV decoupling strategy studied here can be easily extended to other popular time stepping methods and produce a family of unconditionally stable numerical schemes.

The rest of this article is organized as follows. In Section \ref{sec:model}, we briefly introduce the ACNS model. In Section \ref{sec:cnlfac}, the CNLFAC numerical scheme is presented and proved to be unconditionally long time stable. In Section \ref{sec:acsav}, the fully decoupled ACSAV scheme is proposed and proved to be unconditionally stable. Details for its efficient implementation are also provided; Another version of ACSAV that treats the convection term fully explicitly in the NS equations is presented in Section \ref{sec:nsacsav}. In Section \ref{sec:result}, we perform various numerical simulations to demonstrate the stability, feasibility, and computational efficiency of the three proposed algorithms. Some concluding remarks are made in Section \ref{sec:sum}.


\section{The ACNS model}\label{sec:model}
Consider the modeling of a mixture of two immiscible, incompressible fluids in a bounded Lipschitz domain $\Omega$ in $\mathbb{R}^d$ $(d=2, 3)$. 
To label the two different fluids, a phase-field variable (macroscopic labeling function) $\phi$ is introduced, i.e.\begin{equation*}\label{phase-function}
\phi(x,t)=\left\{\begin{aligned}
&1, \qquad  \text{for fluid 1},\\
&-1, \quad \text{for fluid 2},
\end{aligned}\right.
\end{equation*}
with the discontinuity of the function smoothed by a thin, smooth transition region of width $O(\eta)$ ($\eta\ll 1$).
The total energy $W$ of the system is a sum of the kinetic energy and the Ginzburg-Landau type of Helmholtz free energy:
\begin{equation*}\label{energy}
W=\int_\Omega\Big(\frac{1}{2}\lvert u\rvert^2+\lambda\Big(\frac{\lvert\nabla\phi\rvert^2}{2}+F(\phi)\Big)\Big)dx,
\end{equation*}
where $u$ is the fluid velocity, $\lambda$ is related to the surface tension parameter. The second term in the expression of $W$ contributes to the tendency of mixing  between the materials, while the third term, the double-well bulk energy $F(\phi)=\frac{1}{4\eta^2}(\phi^2-1)^2$, represents the tendency of separation. As a consequence of the competition between the two types of interactions, a diffusive interface with thickness proportional to the parameter $\eta$ will form in equilibrium.

Assuming a generalized Ficks law holds
 and the fluid is incompressible, the governing equations of the Allen-Cahn-Navier-Stokes model are derived by minimizing the total energy in the space of $L^2$ \cite{Liu, Pyue}. It writes as
\begin{numcases}{}
	\partial_t\phi+u\cdot\nabla \phi =-M \mu,\label{ACNS-1}\\
	\mu = \frac{\delta W}{\delta \phi}=\lambda (-\Delta \phi  + f(\phi)),\label{ACNS-2}\\
	\partial_t u+ (u\cdot \nabla) u  -\nu \Delta u + \nabla p=\mu \nabla \phi,\label{ACNS-3}\\
	\nabla \cdot u = 0,\label{ACNS-4}
\end{numcases}
\noindent where $f(\phi)= F^{\prime}(\phi)=\frac{(\phi^2-1)\phi}{\eta^2}$, $\mu = \frac{\delta W}{\delta \phi}$ is the variational derivative or chemical potential, $p$ is the pressure, $M$ is the relaxation or mobility parameter of the phase function, and $\nu$ is the viscosity parameter.

 We assume the following boundary conditions for the simplicity of the presentation of the analysis:
$$u|_{\partial \Omega}=0, \quad \partial_{n}\phi |_{\partial\Omega}=0, \quad \partial_n\mu|_{\partial\Omega}=0.$$
The analysis can extended to other boundary conditions with minor modification.

Throughout this paper the $L^2(\Omega)$ norm of scalars, vectors, and tensors will be denoted by $\Vert \cdot\Vert$ with the usual $L^2$ inner product denoted by $(\cdot, \cdot)$. 
By taking the $L^2$ inner product of \eqref{ACNS-1} with $\mu$,  \eqref{ACNS-2} with $-\partial_t\phi$,  \eqref{ACNS-3} with $u$, and then using \eqref{ACNS-4} and summing up the resulted identities we can easily get
\begin{equation*}\label{energydiffusion}
\frac{dW}{dt}=-M\rVert\mu\lVert^2-\nu\lVert \nabla u\rVert^2.
\end{equation*}
This means the total energy of the ACNS system is dissipative.

\section{CNLFAC method for the ACNS model}
\label{sec:cnlfac}
In this section we introduce the Crank--Nicolson leap-frog artificial compression method (CNLFAC) in the semi-discrete form and prove it is unconditionally long time stable. 

In the ACNS equations, the function $f(\phi)=\frac{1}{\eta^2}(\phi^3-\phi)$ is nonlinear and usually results in nonlinear semi-discrete schemes that require Picard/Newton iterations for computation. This makes the computation expensive and takes more simulation time. For differential equations that involve nonlinearity, linear schemes are usually desirable but difficult to design due to stability and convergence issues. For ACNS, only a few fully linear time stepping schemes are available and most of them are conditionally stable. A recent paper by Han et al \cite{Han17} proposed an unconditionally stable, second order, linear scheme for the ACNS equations adopting an idea of Lagrange multiplier of \cite{Guillen13}. Here we adopt the same idea and combine it with an artificial compression method based on the Crank--Nicolson leap-frog time stepping method. With the artificial compression method the computation of the velocity and pressure is decoupled reducing the size of the linear systems to be solved. Moreover, the pressure can be updated directly without solving any partial differential equations further reducing the computational cost. The proposed CNLFAC method is second-order accurate for all three variables: the phase field $\phi$, the fluid velocity $u$, and the fluid pressure $p$.
 
Let $q=\frac{1}{\eta^2}(\phi^2-1)$ and thus $f(\phi)=\phi q$. Taking the derivative of $q$ with respect to $t$ gives $\partial_tq=\frac{2}{\eta^2}\phi\partial_t\phi$, which can be discretized by the Crank--Nicolson leap-frog scheme as 
$$\frac{q^{n+1}-q^{n-1}}{2\Delta t}=\frac{2}{\eta^2} \phi^n \frac{\phi^{n+1}-\phi^{n-1}}{2\Delta t}.$$
By introducing the variable $q$, the total energy can be written as 
\begin{equation}\label{energy-q}
W=\int_\Omega\Big(\frac{1}{2}\lvert u\rvert^2+\lambda\Big(\frac{\lvert\nabla\phi\rvert^2}{2}+\frac{\eta^2}{4}q^2\Big)\Big)dx .
\end{equation}

We will use $q$ as an intermediate variable in our algorithm so that a fully linear scheme could be devised. Let $t_n=n\Delta t$, $n=0,1,2,\cdots, N$, where $N=T/\Delta t$, denote a uniform partition of the interval $[0,T]$. We now propose an unconditionally stable, second order, linear, artificial compression method given by
\begin{Aalgorithm}\label{Algo-ACNS}
Given $u^{n-1}$, $u^n$, $p^{n-1}$, $p^n$, $\phi^{n-1}$, $\phi^n$, $q^{n-1}$, $q^n$, find $u^{n+1}$, $p^{n+1}$, $\phi^{n+1}$ and $q^{n+1}$ satisfying
\begin{align}
&\dot{\phi}^{n+1}-\lambda M \left( \Delta \frac{\phi^{n+1}+\phi^{n-1}}{2}-\phi^n\frac{q^{n+1}+q^{n-1}}{2}\right)=0,\label{ACNS1}\\
&q^{n+1}-q^{n-1}=\frac{2}{\eta^2} \phi^n \left(\phi^{n+1}-\phi^{n-1}\right), \label{ACNS2}\\
&\frac{u^{n+1}-u^{n-1}}{2\Delta t}-\beta \Delta t^{-1}\nabla \left( \nabla \cdot u^{n+1} - \nabla \cdot u^{n-1}\right)\label{ACNS3}\\
&\qquad\qquad\quad+ (u^{n}\cdot\nabla) \left(\frac{u^{n+1}+u^{n-1}}{2}\right)
+\frac{1}{2}\left( \nabla \cdot u^n\right) \left( \frac{u^{n+1}+u^{n-1}}{2}\right)\nonumber\\
&\qquad\qquad\quad-\nu\Delta \left(\frac{u^{n+1}+u^{n-1}}{2}\right)+\nabla p^n +\frac{1}{M}\dot{\phi}^{n+1} \nabla \phi^n=0,\nonumber\\
&\alpha\Delta t\left(p^{n+1}-p^{n-1}\right)+\nabla \cdot u^n=0,\label{ACNS4}
\end{align}
where 
\begin{align}
&\dot{\phi}^{n+1}=\frac{\phi^{n+1}-\phi^{n-1}}{2\Delta t}+\left(\frac{u^{n+1}+u^{n-1}}{2}\right)\cdot\nabla \phi^n, \quad \alpha>0, \quad \beta>0. \nonumber\label{ACNS5}
\end{align}
\end{Aalgorithm}

Note that we could rewrite \eqref{ACNS2} as
\begin{equation}\label{q}
q^{n+1}=q^{n-1}+\frac{2}{\eta^2} \phi^n \left(\phi^{n+1}-\phi^{n-1}\right), 
\end{equation}
and replace $q^{n+1}$ in \eqref{ACNS1} so that \eqref{ACNS1} becomes
\begin{equation*}\label{newphi}
\dot{\phi}^{n+1}-\lambda M \left( \Delta \frac{\phi^{n+1}+\phi^{n-1}}{2}\right)+\lambda M \phi^n \left(\frac{1}{\eta^2} \phi^n \left(\phi^{n+1}-\phi^{n-1}\right)+q^{n-1}\right)=0.
\end{equation*}
Therefore adding the intermediate variable $q$ does not increase the computational cost. One still only needs to solve for the primary variables $u, p$, and $\phi$ while $q$ is updated directly using the formula \eqref{q} in the procedure.  

In this algorithm, the  artificial compression is incorporated by adding a small perturbation of $\partial_tp $ to the mass conservation equation, and discretized as   $2\alpha  \Delta t^2 \cdot\frac{p^{n+1}- p^{n-1} }{2\Delta t}$, see \eqref{ACNS4}. So the scheme has $O(\Delta t^2)$ consistency error. The pressure $p^{n+1}$ can then be updated directly without solving any additional partial differential equations.

\begin{theorem}
[Stability of Algorithm \eqref{Algo-ACNS}]\label{ACNS:stability} Taking $\alpha$ and $\beta$ such that $\alpha\beta\geq \frac{1}{4}$, then for any $N\geq 2$
\begin{align}
\Delta t\sum_{n=1}^{N-1}\frac{1}{M}\Vert \dot{\phi}^{n+1}\Vert^2&+\left(\frac{\lambda}{4}\Vert \nabla \phi^{N}\Vert^2+\frac{\lambda}{4}\Vert \nabla \phi^{N-1}\Vert^2\right)+\frac{\lambda\eta^2}{8} (\Vert q^{N}\Vert^2+\Vert q^{N-1}\Vert^2)\nonumber\\
&+\Delta t\sum_{n=1}^{N-1}\nu\Vert\nabla \frac{u^{n+1}+u^{n-1}}{2}\Vert^2+\frac{1}{4}\left(\Vert u^{N}\Vert^2+\Vert u^{N-1}\Vert^2\right)\nonumber\\
&\leq \left(\frac{\lambda}{4}\Vert \nabla \phi^{1}\Vert^2+\frac{\lambda}{4}\Vert \nabla \phi^{0}\Vert^2\right)
+\frac{\lambda\eta^2}{8} (\Vert q^{1}\Vert^2+\Vert q^{0}\Vert^2)\nonumber\\
&\quad+\frac{1}{4}\left(\Vert u^{1}\Vert^2+2\beta \Vert \nabla \cdot u^{1}\Vert^2\right)+\frac{1}{4}\left(\Vert u^{0}\Vert^2+2\beta \Vert \nabla \cdot u^{0}\Vert^2\right)\nonumber\\
&\quad+\frac{1}{2}\alpha\Delta t^2 \left(\Vert p^{1}\Vert^2+\Vert p^{0}\Vert^2\right)+\frac{1}{2}\Delta t(p^{1}, \nabla \cdot u^{0})- \frac{1}{2}\Delta t(p^{0},\nabla \cdot u^{1}).\nonumber
\end{align}

\end{theorem}

\begin{proof}
Taking the inner product of \eqref{ACNS1} with $\frac{\phi^{n+1}-\phi^{n-1}}{2M \Delta t}$ gives
\begin{align}
\frac{1}{M}\Vert \dot{\phi}^{n+1}\Vert^2&-\frac{1}{M}\left( \dot{\phi}^{n+1}, \frac{u^{n+1}+u^{n-1}}{2}\nabla \phi^n \right)\label{pf1}\nonumber\\
&+\frac{\lambda}{\Delta t}\left(\frac{1}{4}\Vert \nabla \phi^{n+1}\Vert^2-\frac{1}{4}\Vert \nabla \phi^{n-1}\Vert^2\right)\\
&+\frac{\lambda}{2\Delta t}\left(\phi^n\frac{q^{n+1}+q^{n-1}}{2}, \phi^{n+1}-\phi^{n-1}\right)=0.\nonumber
\end{align}
Taking the inner product of \eqref{ACNS2} with $\lambda\eta^2\frac{q^{n+1}+q^{n-1}}{8\Delta t}$ gives
\begin{align}
\frac{\lambda\eta^2}{8\Delta t} \left(\Vert q^{n+1}\Vert^2-\Vert q^{n-1}\Vert^2 \right)= \frac{\lambda}{2\Delta t} \left( \phi^n \left(\phi^{n+1}-\phi^{n-1}\right), \frac{q^{n+1}+q^{n-1}}{2}\right).\label{pf2}
\end{align}
Taking the inner product of \eqref{ACNS3} with $\frac{u^{n+1}+u^{n-1}}{2}$ gives
\begin{align}
\frac{1}{4\Delta t}\left(\Vert u^{n+1}\Vert^2+2\beta \Vert \nabla \cdot u^{n+1}\Vert^2\right)&-\frac{1}{4\Delta t}\left(\Vert u^{n-1}\Vert^2+2\beta \Vert \nabla \cdot u^{n-1}\Vert^2\right)\label{pf3}\\
&+\nu\Vert \nabla\frac{u^{n+1}+u^{n-1}}{2}\Vert^2-\frac{1}{2}\left( p_n, \nabla \cdot u^{n+1}+\nabla \cdot u^{n-1}\right)\nonumber\\
&+\frac{1}{M}\left(\dot{\phi}^{n+1} \nabla \phi^n, \frac{u^{n+1}+u^{n-1}}{2}\right)=0.\nonumber
\end{align}
Taking the inner product of \eqref{ACNS4} with $\frac{p^{n+1}+p^{n-1}}{2}$ yields
\begin{align}
\frac{1}{2}\alpha\Delta t \left(\Vert p^{n+1}\Vert^2-\Vert p^{n-1}\Vert^2\right)+\frac{1}{2}\left(\nabla \cdot u^n, p^{n+1}+p^{n-1}\right)=0.\label{pf4}
\end{align}

By adding \eqref{pf1}, \eqref{pf2}, \eqref{pf3} and \eqref{pf4} we have
\begin{align}
\frac{1}{M}\Vert \dot{\phi}^{n+1}\Vert^2&+\frac{\lambda}{\Delta t}\left(\frac{1}{4}\Vert \nabla \phi^{n+1}\Vert^2+\frac{1}{4}\Vert \nabla \phi^{n}\Vert^2\right)-\frac{\lambda}{\Delta t}\left(\frac{1}{4}\Vert \nabla \phi^{n}\Vert^2+\frac{1}{4}\Vert \nabla \phi^{n-1}\Vert^2\right)\nonumber\\
&+\frac{\lambda\eta^2}{8\Delta t} \left(\Vert q^{n+1}\Vert^2-\Vert q^{n-1}\Vert^2 \right)
+\nu\Vert\nabla \frac{u^{n+1}+u^{n-1}}{2}\Vert^2\label{pf6}\\
&+\frac{1}{4\Delta t}\left(\Vert u^{n+1}\Vert^2+2\beta \Vert \nabla \cdot u^{n+1}\Vert^2\right)-\frac{1}{4\Delta t}\left(\Vert u^{n}\Vert^2+2\beta \Vert \nabla \cdot u^{n}\Vert^2\right)\nonumber\\
&+\frac{1}{4\Delta t}\left(\Vert u^{n}\Vert^2+2\beta \Vert \nabla \cdot u^{n}\Vert^2\right)-\frac{1}{4\Delta t}\left(\Vert u^{n-1}\Vert^2+2\beta \Vert \nabla \cdot u^{n-1}\Vert^2\right)\nonumber\\
&+\frac{1}{2}\alpha\Delta t \left(\Vert p^{n+1}\Vert^2-\Vert p^{n}\Vert^2\right)+\frac{1}{2}\alpha\Delta t \left(\Vert p^{n}\Vert^2-\Vert p^{n-1}\Vert^2\right)\nonumber\\
&+\frac{1}{2}\left(\nabla \cdot u^n, p^{n+1}+p^{n-1}\right)-\frac{1}{2}\left( p^n, \nabla\cdot u^{n+1}+\nabla \cdot u^{n-1}\right)=0.\nonumber
\end{align}
The last two terms of \eqref{pf6} can be rewritten as
\begin{align}
&\frac{1}{2}\left(\nabla \cdot u^n, p^{n+1}+p^{n-1}\right)-\frac{1}{2}\left( p^n, \nabla\cdot u^{n+1}+\nabla \cdot u^{n-1}\right)\nonumber\\
&=\frac{1}{2}\left[(p^{n+1}, \nabla \cdot u^n)-(p^n, \nabla \cdot u^{n-1})\right]-\frac{1}{2}\left[ (p^n,\nabla \cdot u^{n+1})-(p^{n-1}, \nabla \cdot u^n)\right].\nonumber
\end{align}
Then summing up \eqref{pf6} from $n=1$ to $n=N-1$ and multiplying through by $\Delta t$ gives
\begin{align}
\Delta t\sum_{n=1}^{N-1}\frac{1}{M}\Vert \dot{\phi}^{n+1}\Vert^2&+\left(\frac{\lambda}{4}\Vert \nabla \phi^{N}\Vert^2+\frac{\lambda}{4}\Vert \nabla \phi^{N-1}\Vert^2\right)+\frac{\lambda\eta^2}{8} (\Vert q^{N}\Vert^2+\Vert q^{N-1}\Vert^2)\nonumber\\
&+\Delta t\sum_{n=1}^{N-1}\nu\Vert\nabla \frac{u^{n+1}+u^{n-1}}{2}\Vert^2+\frac{1}{4}\left(\Vert u^{N}\Vert^2+2\beta \Vert \nabla \cdot u^{N}\Vert^2\right)\nonumber\\
&+\frac{1}{4}\left(\Vert u^{N-1}\Vert^2+2\beta \Vert \nabla \cdot u^{N-1}\Vert^2\right)+\frac{1}{2}\alpha\Delta t^2 \left(\Vert p^{N}\Vert^2+\Vert p^{N-1}\Vert^2\right)\nonumber\\
&+\frac{1}{2}\Delta t(p^{N}, \nabla \cdot u^{N-1})- \frac{1}{2}\Delta t(p^{N-1},\nabla \cdot u^{N})\label{pf7}\\
&=\left(\frac{\lambda}{4}\Vert \nabla \phi^{1}\Vert^2+\frac{\lambda}{4}\Vert \nabla \phi^{0}\Vert^2\right)
+\frac{\lambda\eta^2}{8} (\Vert q^{1}\Vert^2+\Vert q^{0}\Vert^2)\nonumber\\
&\quad+\frac{1}{4}\left(\Vert u^{1}\Vert^2+2\beta \Vert \nabla \cdot u^{1}\Vert^2\right)+\frac{1}{4}\left(\Vert u^{0}\Vert^2+2\beta \Vert \nabla \cdot u^{0}\Vert^2\right)\nonumber\\
&\quad+\frac{1}{2}\alpha\Delta t^2 \left(\Vert p^{1}\Vert^2+\Vert p^{0}\Vert^2\right)+\frac{1}{2}\Delta t(p^{1}, \nabla \cdot u^{0})- \frac{1}{2}\Delta t(p^{0},\nabla \cdot u^{1}).\nonumber
\end{align}

\noindent The last two terms on the left hand side of \eqref{pf7} can be bounded as, for $\forall \beta >0$,
\begin{align}
&\Bigg\vert \frac{1}{2}\Delta t(p^{N}, \nabla \cdot u^{N-1})-\frac{1}{2} \Delta t(p^{N-1},\nabla \cdot u^{N}) \Bigg\vert
\\
&\quad\leq \frac{\beta}{2} \Vert \nabla \cdot u^{N-1}\Vert^2 +\frac{1}{8\beta}\Delta t^2\Vert p^N\Vert^2+\frac{\beta}{2} \Vert \nabla \cdot u^{N}\Vert^2 +\frac{1}{8\beta}\Delta t^2\Vert p^{N-1}\Vert^2.\nonumber
\end{align}
So if $\alpha\geq \frac{1}{4\beta}$, \eqref{pf7} reduces to
\begin{align}
&\Delta t\sum_{n=1}^{N-1}\frac{1}{M}\Vert \dot{\phi}^{n+1}\Vert^2
+\left(\frac{\lambda}{4}\Vert \nabla \phi^{N}\Vert^2+\frac{\lambda}{4}\Vert \nabla \phi^{N-1}\Vert^2\right)
\\
&
+\frac{\lambda\eta^2}{8} (\Vert q^{N}\Vert^2+\Vert q^{N-1}\Vert^2)
+\Delta t\sum_{n=1}^{N-1}\nu\Vert\nabla \frac{u^{n+1}+u^{n-1}}{2}\Vert^2+\frac{1}{4}\left(\Vert u^{N}\Vert^2+\Vert u^{N-1}\Vert^2\right)\nonumber\\
&\leq \left(\frac{\lambda}{4}\Vert \nabla \phi^{1}\Vert^2+\frac{\lambda}{4}\Vert \nabla \phi^{0}\Vert^2\right)
+\frac{\lambda\eta^2}{8} (\Vert q^{1}\Vert^2+\Vert q^{0}\Vert^2)\nonumber\\
&\quad+\frac{1}{4}\left(\Vert u^{1}\Vert^2+2\beta \Vert \nabla \cdot u^{1}\Vert^2\right)+\frac{1}{4}\left(\Vert u^{0}\Vert^2+2\beta \Vert \nabla \cdot u^{0}\Vert^2\right)\nonumber\\
&\quad+\frac{1}{2}\alpha\Delta t^2 \left(\Vert p^{1}\Vert^2+\Vert p^{0}\Vert^2\right)+\frac{1}{2}\Delta t(p^{1}, \nabla \cdot u^{0})- \frac{1}{2}\Delta t(p^{0},\nabla \cdot u^{1}).\nonumber
\end{align}

\end{proof}

\section{ACSAV method for the ACNS model}\label{sec:acsav}
Inspired by the new decoupling strategy proposed in \cite{JY21b}, we now proceed to incorporate the artificial compression technique with an SAV decoupling strategy for developing a highly efficient, fully decoupled scheme.  
This is achieved by introducing a scalar auxiliary variable to handle the lagged coupling terms in the ACNS model, namely the phase induced stress term in the NS equations and the fluid induced transport term in the Allen--Cahn equations. An artificial compression based scalar auxiliary variable scheme (ACSAV) is then derived and proved to be long time stable without any time step constraints.

Define a scalar auxiliary variable $r(t)$ by
\begin{align}
r(t)=exp(-\frac{t}{T}).\nonumber
\end{align}
Then we have
\begin{align}\label{drdt0}
&
\frac{dr}{dt}=-\frac{1}{T} r + \frac{1}{exp(-\frac{t}{T})} \int_{\Omega} \Big((u \cdot \nabla\phi)\mu-(u \cdot \nabla\phi)\mu \Big)\, dx .
\end{align}
To decouple the phase field variable $\phi$ and fluid velocity $u$, one usually needs to lag the coupling terms $u \cdot \nabla\phi$ in \eqref{ACNS-1}  and $\mu\nabla\phi$ in \eqref{ACNS-3} to the previous time steps and this inevitably results in a time step condition to ensure long time stability. The introduction of the zero term $\displaystyle\int_{\Omega} \Big((u \cdot \nabla\phi)\mu-(u \cdot \nabla\phi)\mu \Big)\, dx$ in \eqref{drdt0} makes it possible to cancel out the same term in the partitioned scheme with lagged coupling terms. 

The governing equations of the Allen-Cahn-Navier-Stokes (ACNS) system are then equivalent to:
\begin{numcases}{}
\partial_t\phi+\frac{r(t)}{exp(-\frac{t}{T})} u\cdot\nabla \phi = -M\mu,\label{eq:ACSAV-AC1}\\
\mu = \lambda(-\Delta \phi + q\phi),\label{eq:ACSAV-AC2}\\
\partial_t q = \tfrac{2}{\eta^2}\phi\partial_t\phi,\label{eq:ACSAV-AC3}\\
\partial_tu+ (u\cdot \nabla) u  -\nu \Delta u + \nabla p = \frac{r(t)}{exp(-\frac{t}{T})}\mu\nabla\phi,\label{eq:ACSAV-NS1}\\
\nabla \cdot u = 0,\label{eq:ACSAV-NS2}\\
\frac{dr}{dt}=-\frac{1}{T} r + \frac{1}{exp(-\frac{t}{T})} \int_{\Omega} \Big((u \cdot \nabla\phi)\mu-(u \cdot \nabla\phi)\mu \Big)\, dx .\label{eq:ACSAV-drdt}
\end{numcases}

We now propose an unconditionally stable, second order, linear, fully decoupled ACSAV method given by
\begin{Aalgorithm}\label{Algo-ACNS-ACSAV}
Given $u^{n-1}$, $u^n$, $p^{n-1}$, $p^n$, $\phi^{n-2}$, $\phi^{n-1}$, $\phi^n$, $q^{n-1}$, $q^n$, $r^{n-1}$, $r^n$, $\alpha>0$, $\beta>0$, find $u^{n+1}$, $p^{n+1}$, $\phi^{n+1}$, $q^{n+1}$, and $r^{n+1}$ satisfying
\begin{align}
&\frac{\phi^{n+1}-\phi^{n-1}}{2\Delta t}+\frac{r^{n+1}+r^{n-1}}{2exp(-\frac{t^n}{T})} u^n\cdot\nabla \phi^n = -M\bar{\mu}^n,\label{ACNS-ACSAV1}\\
&\bar{\mu}^n = \lambda(-\Delta\frac{\phi^{n+1}+\phi^{n-1}}{2}+\frac{q^{n+1}+q^{n-1}}{2}\phi^n),\label{mubareq}\\
&\frac{q^{n+1}-q^{n-1}}{2\Delta t}=\frac{2}{\eta^2} \phi^n \frac{\phi^{n+1}-\phi^{n-1}}{2\Delta t}, \label{ACNS-ACSAV2}\\
&\frac{u^{n+1}-u^{n-1}}{2\Delta t}+ (u^{n}\cdot\nabla) \left(\frac{u^{n+1}+u^{n-1}}{2}\right)\notag\\
& \qquad+\frac{1}{2}(\nabla\cdot u^n)\left(\frac{u^{n+1}+u^{n-1}}{2}\right) -\beta \Delta t^{-1}\nabla \left( \nabla \cdot u^{n+1} - \nabla \cdot u^{n-1}\right)\label{ACNS-ACSAV3}\\
&\qquad-\nu\Delta \left(\frac{u^{n+1}+u^{n-1}}{2}\right)+\nabla p^n =\frac{r^{n+1}+r^{n-1}}{2exp(-\frac{t^n}{T})} \mu^n\nabla\phi^n ,\nonumber\\
&\mu^n = -\frac{1}{M}\left( \frac{3\phi^n-4\phi^{n-1}+\phi^{n-2}}{2\Delta t}+\frac{r^n}{exp(-\frac{t^n}{T})}u^n\cdot\nabla\phi^n \right),\label{mueq}\\
&\alpha\Delta t\left(p^{n+1}-p^{n-1}\right)+\nabla \cdot u^n=0,\label{ACNS-ACSAV4}\\
&\frac{r^{n+1}-r^{n-1}}{2\Delta t}
= -\frac{1}{T} \frac{r^{n+1}+r^{n-1}}{2}+ \frac{1}{exp(-\frac{t^{n}}{T})} \int_{\Omega}( u^{n} \cdot \nabla \phi^{n}) \bar{\mu}^n  \, dx \label{drdt1}\\
&\qquad - \frac{1}{exp(-\frac{t^{n}}{T})} \int_{\Omega} (\frac{u^{n+1}+u^{n-1}}{2} \cdot \nabla \phi^{n}) \cdot\mu^n \, dx \notag.
\end{align}

\end{Aalgorithm}

Note that we could rewrite \eqref{ACNS-ACSAV2} as \eqref{q}
and replace $q^{n+1}$ in \eqref{mubareq} so that \eqref{mubareq} becomes
\begin{equation}\label{newmubar}
\bar{\mu}^n = \lambda(-\Delta\frac{\phi^{n+1}+\phi^{n-1}}{2})+\lambda\Big(q^{n-1}+\frac{1}{\eta^2} \phi^n \left(\phi^{n+1}-\phi^{n-1}\right)\Big)\phi^n.
\end{equation}
Therefore, as in the CNLFAC scheme, adding the intermediate variable $q$ does not increase the computational cost. Although the variables $\phi^{n+1}$ and $u^{n+1}$ are coupled with $r^{n+1}$ in Algorithm \ref{Algo-ACNS-ACSAV}, we will present later that an efficient splitting procedure can be employed to separate the computation of $\phi^{n+1}$ from $r^{n+1}$ and $u^{n+1}$ from $r^{n+1}$ resulting in a fully decoupled  scheme for computing $\phi^{n+1}$, $u^{n+1}$, $p^{n+1}$.

\begin{theorem}\label{acsavthm}
[Stability of Algorithm \eqref{Algo-ACNS-ACSAV}] Taking $\alpha$ and $\beta$ such that $\alpha\beta\geq \frac{1}{4}$, then for any $N\geq 3$
\begin{align}
\Delta tM\sum_{n=2}^{N-1}\Vert \bar{\mu}^n\Vert^2&+\left(\frac{\lambda}{4}\Vert \nabla \phi^{N}\Vert^2+\frac{\lambda}{4}\Vert \nabla \phi^{N-1}\Vert^2\right)+\frac{\lambda\eta^2}{8} (\Vert q^{N}\Vert^2+\Vert q^{N-1}\Vert^2)\nonumber\\
&
+\frac{1}{4}\Vert u^{N}\Vert^2+\frac{1}{4}\Vert u^{N-1}\Vert^2
+\Delta t\sum_{n=2}^{N-1}\nu\Vert\nabla \frac{u^{n+1}+u^{n-1}}{2}\Vert^2
\notag\\
&+\frac{1}{4}(|r^{N}|^2+|r^{N-1}|^2)+\frac{\Delta t}{T}\sum\limits_{n=2}^{N-1}|\frac{r^{n+1}+r^{n-1}}{2}|^2 \label{acsavineq}\\
&\leq \left(\frac{\lambda}{4}\Vert \nabla \phi^{2}\Vert^2+\frac{\lambda}{4}\Vert \nabla \phi^{1}\Vert^2\right)
+\frac{\lambda\eta^2}{8} (\Vert q^{2}\Vert^2+\Vert q^{1}\Vert^2)\nonumber\\
&\quad+\frac{1}{4}\left(\Vert u^{2}\Vert^2+2\beta \Vert \nabla \cdot u^{2}\Vert^2\right)+\frac{1}{4}\left(\Vert u^{1}\Vert^2+2\beta \Vert \nabla \cdot u^{1}\Vert^2\right)\nonumber\\
&\quad+\frac{1}{2}\alpha\Delta t^2 \left(\Vert p^{2}\Vert^2+\Vert p^{1}\Vert^2\right) + \frac{1}{4}(|r^{2}|^2+|r^{1}|^2)\nonumber\\
&\quad +\frac{1}{2}\Delta t(p^{2}, \nabla \cdot u^{1})- \frac{1}{2}\Delta t(p^{1},\nabla \cdot u^{2}).\nonumber
\end{align}

\end{theorem}

\begin{proof}
Taking the inner product of \eqref{ACNS-ACSAV1} with $\bar{\mu}^n$, the inner product of \eqref{mubareq} with $-\frac{\phi^{n+1}-\phi^{n-1}}{2\Delta t}$, the inner product of \eqref{ACNS-ACSAV2} with $\frac{\lambda\eta^2}{2}\frac{q^{n+1}+q^{n-1}}{2}$ and then adding the equations together gives
\begin{align}
\frac{r^{n+1}+r^{n-1}}{2exp(-\frac{t^n}{T})} (u^n\cdot\nabla \phi^n,\bar{\mu}^n) &+ M\|\bar{\mu}^n\|^2 + \frac{\lambda}{4\Delta t}(\|\nabla\phi^{n+1}\|^2-\|\nabla\phi^{n-1}\|^2) \notag\\
&+ \frac{\lambda\eta^2}{8\Delta t}(\|q^{n+1}\|^2-\|q^{n-1}\|^2) = 0 \label{npf1}.
\end{align}
Taking the inner product of \eqref{ACNS-ACSAV3} with $\frac{u^{n+1}+u^{n-1}}{2} $ gives
\begin{align}
\frac{1}{4\Delta t}\left(\Vert u^{n+1}\Vert^2+2\beta \Vert \nabla \cdot u^{n+1}\Vert^2\right)&-\frac{1}{4\Delta t}\left(\Vert u^{n-1}\Vert^2+2\beta \Vert \nabla \cdot u^{n-1}\Vert^2\right)\label{npf2}\\
&+\nu\Vert \nabla\frac{u^{n+1}+u^{n-1}}{2}\Vert^2-\frac{1}{2}\left( p^n, \nabla \cdot u^{n+1}+\nabla \cdot u^{n-1}\right)\notag\\
&= \frac{r^{n+1}+r^{n-1}}{2exp(-\frac{t^n}{T})}(\mu^n\nabla\phi^n,\frac{u^{n+1}+u^{n-1}}{2}).\nonumber
\end{align}
Taking the inner product of \eqref{ACNS-ACSAV4} with $\frac{p^{n+1}+p^{n-1}}{2}$ yields
\begin{align}
\frac{1}{2}\alpha\Delta t \left(\Vert p^{n+1}\Vert^2-\Vert p^{n-1}\Vert^2\right)+\frac{1}{2}\left(\nabla \cdot u^n, p^{n+1}+p^{n-1}\right)=0.\label{npf3}
\end{align}
Taking the product of \eqref{drdt1} with $\frac{r^{n+1}+r^{n-1}}{2}$ yields
\begin{align}
\frac{1}{4\Delta t}(|r^{n+1}|^2-|r^{n-1}|^2)&+\frac{1}{T}|\frac{r^{n+1}+r^{n-1}}{2}|^2\nonumber\\
&= \frac{r^{n+1}+r^{n-1}}{2exp(-\frac{t^n}{T})} \int_{\Omega}( u^{n} \cdot \nabla \phi^{n}) \bar{\mu}^n  \, dx\nonumber \\
&\quad- \frac{r^{n+1}+r^{n-1}}{2exp(-\frac{t^n}{T})}\int_{\Omega} (\frac{u^{n+1}+u^{n-1}}{2} \cdot \nabla \phi^{n}) \cdot\mu^n \, dx. \label{npf4}
\end{align}

By adding \eqref{npf1}, \eqref{npf2}, \eqref{npf3} and \eqref{npf4} we have
\begin{equation}
\begin{split}\label{npf5}
M\|\bar{\mu}^n\|^2 &+ \frac{\lambda}{4\Delta t}(\|\nabla\phi^{n+1}\|^2-\|\nabla\phi^{n-1}\|^2)  + \frac{\lambda\eta^2}{8\Delta t}(\|q^{n+1}\|^2-\|q^{n-1}\|^2)  \\
&+\frac{1}{4\Delta t}\left(\Vert u^{n+1}\Vert^2+2\beta \Vert \nabla \cdot u^{n+1}\Vert^2\right)-\frac{1}{4\Delta t}\left(\Vert u^{n-1}\Vert^2+2\beta \Vert \nabla \cdot u^{n-1}\Vert^2\right)\\
&+\nu\Vert \nabla\frac{u^{n+1}+u^{n-1}}{2}\Vert^2 + \frac{1}{2}\alpha\Delta t \left(\Vert p^{n+1}\Vert^2-\Vert p^{n-1}\Vert^2\right)\\
&+\frac{1}{4\Delta t}(|r^{n+1}|^2-|r^{n-1}|^2)+\frac{1}{T}|\frac{r^{n+1}+r^{n-1}}{2}|^2 \\
& +\frac{1}{2}\left(\nabla \cdot u^n, p^{n+1}+p^{n-1}\right) - \frac{1}{2}\left( p^n, \nabla \cdot u^{n+1}+\nabla \cdot u^{n-1}\right)=0.
\end{split}
\end{equation}
The last two terms on the left hand side of \eqref{npf5} can be rewritten as
\begin{align}
\frac{1}{2}&\left(\nabla \cdot u^n, p^{n+1}+p^{n-1}\right) - \frac{1}{2} \left( p^n, \nabla\cdot u^{n+1}+\nabla \cdot u^{n-1}\right)\nonumber \\
&=\frac{1}{2}\left[(p^{n+1}, \nabla \cdot u^n)-(p^n, \nabla \cdot u^{n-1})\right] -\frac{1}{2}\left[ (p^n,\nabla \cdot u^{n+1})-(p^{n-1}, \nabla \cdot u^n)\right] .\nonumber
\end{align}
Then summing up \eqref{npf5} from $n=2$ to $n=N-1$ and multiplying through by $\Delta t$ gives
\begin{align}
\Delta tM\sum_{n=2}^{N-1}\Vert \bar{\mu}^n\Vert^2&+\left(\frac{\lambda}{4}\Vert \nabla \phi^{N}\Vert^2+\frac{\lambda}{4}\Vert \nabla \phi^{N-1}\Vert^2\right)+\frac{\lambda\eta^2}{8} (\Vert q^{N}\Vert^2+\Vert q^{N-1}\Vert^2)\nonumber\\
&
+\frac{1}{4}\left(\Vert u^{N}\Vert^2+2\beta \Vert \nabla \cdot u^{N}\Vert^2\right)+\frac{1}{4}\left(\Vert u^{N-1}\Vert^2+2\beta \Vert \nabla \cdot u^{N-1}\Vert^2\right)\nonumber\\
&
+\Delta t\sum_{n=2}^{N-1}\nu\Vert\nabla \frac{u^{n+1}+u^{n-1}}{2}\Vert^2
+\frac{1}{2}\alpha\Delta t^2 \left(\Vert p^{N}\Vert^2+\Vert p^{N-1}\Vert^2\right)\nonumber\\
&+\frac{1}{4}(|r^{N}|^2+|r^{N-1}|^2)+\frac{\Delta t}{T}\sum\limits_{n=2}^{N-1}|\frac{r^{n+1}+r^{n-1}}{2}|^2 \label{npf7}\\
&+\frac{1}{2}\Delta t(p^{N}, \nabla \cdot u^{N-1})- \frac{1}{2}\Delta t(p^{N-1},\nabla \cdot u^{N})\nonumber\\
&=\left(\frac{\lambda}{4}\Vert \nabla \phi^{2}\Vert^2+\frac{\lambda}{4}\Vert \nabla \phi^{1}\Vert^2\right)
+\frac{\lambda\eta^2}{8} (\Vert q^{2}\Vert^2+\Vert q^{1}\Vert^2)\nonumber\\
&\quad+\frac{1}{4}\left(\Vert u^{2}\Vert^2+2\beta \Vert \nabla \cdot u^{2}\Vert^2\right)+\frac{1}{4}\left(\Vert u^{1}\Vert^2+2\beta \Vert \nabla \cdot u^{1}\Vert^2\right)\nonumber\\
&\quad+\frac{1}{2}\alpha\Delta t^2 \left(\Vert p^{2}\Vert^2+\Vert p^{1}\Vert^2\right) + \frac{1}{4}(|r^{2}|^2+|r^{1}|^2) \nonumber\\
&\quad+\frac{1}{2}\Delta t(p^{2}, \nabla \cdot u^{1})- \frac{1}{2}\Delta t(p^{1},\nabla \cdot u^{2}).\nonumber
\end{align}

The last two terms on the left hand side of \eqref{npf7} can be bounded as
\begin{align}
\frac{1}{2}\Delta t&(p^{N}, \nabla \cdot u^{N-1})- \frac{1}{2}\Delta t(p^{N-1},\nabla \cdot u^{N})\nonumber\\
&\leq \frac{\beta}{2} \Vert \nabla \cdot u^{N-1}\Vert^2 +\frac{1}{8\beta}\Delta t^2\Vert p^N\Vert^2+\frac{\beta}{2} \Vert \nabla \cdot u^{N}\Vert^2 +\frac{1}{8\beta}\Delta t^2\Vert p^{N-1}\Vert^2.\nonumber
\end{align}
So if $\alpha\geq \frac{1}{4\beta}$, \eqref{npf7} reduces to
\begin{align}
\Delta tM\sum_{n=2}^{N-1}\Vert \bar{\mu}^n\Vert^2&+\left(\frac{\lambda}{4}\Vert \nabla \phi^{N}\Vert^2+\frac{\lambda}{4}\Vert \nabla \phi^{N-1}\Vert^2\right)+\frac{\lambda\eta^2}{8} (\Vert q^{N}\Vert^2+\Vert q^{N-1}\Vert^2)\nonumber\\
&
+\frac{1}{4}\Vert u^{N}\Vert^2+\frac{1}{4}\Vert u^{N-1}\Vert^2
+\Delta t\sum_{n=2}^{N-1}\nu\Vert\nabla \frac{u^{n+1}+u^{n-1}}{2}\Vert^2
\label{npf8}\nonumber\\
&+\frac{1}{4}(|r^{N}|^2+|r^{N-1}|^2)+\frac{\Delta t}{T}\sum\limits_{n=2}^{N-1}|\frac{r^{n+1}+r^{n-1}}{2}|^2 \notag\\
&\leq \left(\frac{\lambda}{4}\Vert \nabla \phi^{2}\Vert^2+\frac{\lambda}{4}\Vert \nabla \phi^{1}\Vert^2\right)
+\frac{\lambda\eta^2}{8} (\Vert q^{2}\Vert^2+\Vert q^{1}\Vert^2)\nonumber\\
&\quad+\frac{1}{4}\left(\Vert u^{2}\Vert^2+2\beta \Vert \nabla \cdot u^{2}\Vert^2\right)+\frac{1}{4}\left(\Vert u^{1}\Vert^2+2\beta \Vert \nabla \cdot u^{1}\Vert^2\right)\nonumber\\
&\quad+\frac{1}{2}\alpha\Delta t^2 \left(\Vert p^{2}\Vert^2+\Vert p^{1}\Vert^2\right) + \frac{1}{4}(|r^{2}|^2+|r^{1}|^2)\nonumber\\
&\quad +\frac{1}{2}\Delta t(p^{2}, \nabla \cdot u^{1})- \frac{1}{2}\Delta t(p^{1},\nabla \cdot u^{2}).\nonumber
\end{align}

\end{proof}

\subsection*{Implementation of the ACSAV algorithm} As mentioned before, in the original form of Algorithm \ref{Algo-ACNS-ACSAV}, the variables $\phi^{n+1}$ and $u^{n+1}$ are coupled with $r^{n+1}$. Here we present an efficient splitting procedure to separate the computation of $\phi^{n+1}$ from $r^{n+1}$ and $u^{n+1}$ from $r^{n+1}$.
We will introduce a new scalar $V^{n+1}$ to decompose $\phi^{n+1}$ into two parts yielding two subproblems for the two components $\hat{\phi}^{n+1}$, $\breve{\phi}^{n+1}$ respectively, which do not contain $r^{n+1}$. Similarly, we do a decomposition for $u^{n+1}$. A separate equation for updating $V^{n+1}$ (hence $r^{n+1}$) will be derived.
 
Let
\begin{equation}\label{Vn+1}
V^{n+1}=\frac{r^{n+1}+r^{n-1}}{2exp(-\frac{t^n}{T})},~ \phi^{n+1}=\hat{\phi}^{n+1}+V^{n+1}\breve{\phi}^{n+1},~ u^{n+1}=\hat{u}^{n+1}+V^{n+1}\breve{u}^{n+1}.
\end{equation}
Then instead of solving \eqref{ACNS-ACSAV1}, \eqref{newmubar}, and \eqref{drdt1}, we solve the following two subproblems for $\hat{\phi}^{n+1}$, $\breve{\phi}^{n+1}$ respectively.
\begin{equation}\tag{{ACSAV Subproblem 1}}
\begin{split}
&
\frac{\hat{\phi}^{n+1}-\phi^{n-1}}{2\Delta t}-\lambda M \left( \Delta \frac{\hat{\phi}^{n+1}+\phi^{n-1}}{2}\right)\\
&\qquad\qquad\quad
+\lambda M \phi^n \left(\frac{1}{\eta^2} \phi^n \left(\hat{\phi}^{n+1}-\phi^{n-1}\right)+q^{n-1}\right)=0.
\end{split}
\end{equation}
\begin{equation}\tag{{ACSAV Subproblem 2}}
\frac{\breve{\phi}^{n+1}}{2\Delta t}+ u^n\cdot\nabla \phi^n-\lambda M \left( \Delta \frac{\breve{\phi}^{n+1}}{2}\right)+\lambda M \phi^n \left(\frac{1}{\eta^2} \phi^n \breve{\phi}^{n+1}\right)=0.
\end{equation}
Instead of solving \eqref{ACNS-ACSAV3}, \eqref{mueq}, and \eqref{drdt1}, we solve the following two subproblems for $\hat{u}^{n+1}$, $\breve{u}^{n+1}$ respectively.
\begin{equation}\tag{{ACSAV Subproblem 3}}
\begin{split}
\frac{\hat{u}^{n+1}-u^{n-1}}{2\Delta t} &+ (u^n\cdot\nabla)\left(\frac{\hat{u}^{n+1}+u^{n-1}}{2}\right) + \frac{1}{2}(\nabla\cdot u^n)\left(\frac{\hat{u}^{n+1}+u^{n-1}}{2}\right)\\
&-\beta \Delta t^{-1}\nabla \left( \nabla \cdot \hat{u}^{n+1} - \nabla \cdot u^{n-1}\right)-\nu\Delta \left(\frac{\hat{u}^{n+1}+u^{n-1}}{2}\right)+\nabla p^n = 0.
\end{split}
\end{equation}
\begin{equation}\tag{{ACSAV Subproblem 4}}
\begin{split}
&
\frac{\breve{u}^{n+1}}{2\Delta t} + (u^n\cdot\nabla)\frac{\breve{u}^{n+1}}{2} + \frac{1}{2}(\nabla\cdot u^n)\frac{\breve{u}^{n+1}}{2}\\
&\qquad
-\beta \Delta t^{-1}\nabla \left( \nabla \cdot \breve{u}^{n+1}\right)-\nu\Delta \left(\frac{\breve{u}^{n+1}}{2}\right) = \mu^n\nabla\phi^n.
\end{split}
\end{equation}

Now we can derive an equation for $V^{n+1}$. From \eqref{Vn+1}, we have
\begin{align}
r^{n+1}=2exp(-\frac{t^{n}}{T})V^{n+1} - r^{n-1}. \label{req}
\end{align}
From the decomposition of $\phi^{n+1}$ and equation \eqref{ACNS-ACSAV1} we have the splitting of $\bar{\mu}^n$ as $$\bar{\mu}^n = \hat{\mu}^{n+1} + V^{n+1}\breve{\mu}^{n+1},$$
where
\begin{align}
& \hat{\mu}^{n+1} = -\frac{1}{M}\left(\frac{\hat{\phi}^{n+1}-\phi^{n-1}}{2\Delta t}\right),\label{muhat}\\
& \breve{\mu}^{n+1} = -\frac{1}{M}\left(\frac{\breve{\phi}^{n+1}}{2\Delta t}+u^n\cdot\nabla\phi^n\right). \label{mubre}
\end{align}
Plugging the expression \eqref{req} of $r^{n+1}$ into \eqref{npf4} gives
\begin{align*}
&(\frac{1}{\Delta t}+\frac{1}{T})exp(-\frac{2t^n}{T})(V^{n+1})^2-\frac{1}{\Delta t}exp(-\frac{t^n}{T})r^{n-1}V^{n+1}\\
& =(V^{n+1})^2 \int_{\Omega}( u^{n} \cdot \nabla \phi^{n}) \breve{\mu}^{n+1}  \, dx  +V^{n+1} \int_{\Omega}( u^{n} \cdot \nabla \phi^{n}) \hat{\mu}^{n+1}  \, dx \notag\\
&\quad- (V^{n+1})^2 \int_{\Omega} (\frac{\breve{u}^{n+1}}{2} \cdot \nabla \phi^{n}) \cdot\mu^n \, dx - V^{n+1} \int_{\Omega} (\frac{\hat{u}^{n+1}+u^{n-1}}{2} \cdot \nabla \phi^{n}) \cdot\mu^n \, dx .\notag
\end{align*}
Then we can compute $V^{n+1}$ by solving
\begin{align}\label{V^n+1}
A^{n+1}V^{n+1} + B^{n+1} = 0,
\end{align}
where
\begin{align}
& A^{n+1} = 
(\frac{1}{\Delta t}+\frac{1}{T})exp(-\frac{2t^{n}}{T}) 
 -\int_{\Omega}( u^{n} \cdot \nabla \phi^{n}) \breve{\mu}^{n+1}  \, dx
    +\int_{\Omega} (\frac{\breve{u}^{n+1}}{2} \cdot \nabla \phi^{n}) \cdot\mu^n \, dx,\label{An+1}\\
& B^{n+1} =  -\frac{r^{n-1}}{\Delta t} exp(-\frac{t^{n}}{T})- 
\int_{\Omega}( u^{n} \cdot \nabla \phi^{n}) \hat{\mu}^{n+1}  \, dx
+ \int_{\Omega} (\frac{\hat{u}^{n+1}+u^{n-1}}{2} \cdot \nabla \phi^{n}) \cdot\mu^n \, dx.
\end{align}

\begin{theorem}
The scalar equation \eqref{V^n+1} admits a unique solution.
\end{theorem}
\begin{proof}
Taking the inner product of \eqref{mubre} with $\breve{\mu}^{n+1}$, we get 
\begin{equation}\label{tmp1}
-\int_{\Omega}( u^{n} \cdot \nabla \phi^{n}) \breve{\mu}^{n+1}  \, dx 
= M\|\breve{\mu}^{n+1}\|^2 + \frac{1}{2\Delta t}(\breve{\phi}^{n+1}, \breve{u}^{n+1}).
\end{equation}
Plugging \eqref{mubre} into ACSAV Subproblem 2, and taking its inner product with $\breve{\phi}^{n+1}$, we obtain
\begin{equation}\label{tmp2}
(\breve{\phi}^{n+1},\breve{\mu}^{n+1}) = \frac{\lambda}{2}\|\nabla\breve{\phi}^{n+1}\|^2 + \frac{\lambda}{\eta^2}\|\phi^n\breve{\phi}^{n+1}\|^2.
\end{equation}
Adding \eqref{tmp1} and \eqref{tmp2}, we have
\begin{equation}\label{tmp3}
-\int_{\Omega}( u^{n} \cdot \nabla \phi^{n}) \breve{\mu}^{n+1}  \, dx 
= M\|\breve{\mu}^{n+1}\|^2 + \frac{\lambda}{4\Delta t}\|\nabla\breve{\phi}^{n+1}\|^2 + \frac{\lambda}{2\eta^2\Delta t}\|\phi^n\breve{\phi}^{n+1}\|^2 \geq 0.
\end{equation}
Taking the inner product of ACSAV Subproblem 4 with $\frac{\breve{u}^{n+1}}{2}$, we get 
\begin{equation}\label{tmp4}
\int_{\Omega} (\frac{\breve{u}^{n+1}}{2} \cdot \nabla \phi^{n}) \cdot\mu^n \, dx = \frac{1}{4\Delta t}\|\breve{u}^{n+1}\|^2 + \frac{\beta}{2\Delta t}\|\nabla\cdot\breve{u}^{n+1}\|^2 + \frac{\nu}{4}\|\nabla\breve{u}^{n+1}\|^2 \geq 0.
\end{equation}
From \eqref{tmp3}, \eqref{tmp4} and the expression of $A^{n+1}$ in \eqref{An+1}, we conclude that
$$A^{n+1} > 0.$$
Therefore, the scalar equation $A^{n+1}V^{n+1} + B^{n+1} = 0$ admits a unique solution.
\end{proof}

To summarize, Algorithm \ref{Algo-ACNS-ACSAV} can be implemented in the following way:
\begin{itemize}
	\item\emph{Step 1}: Solve $\hat{\phi}^{n+1}$ and $\breve{\phi}^{n+1}$ from ACSAV Subproblem 1 and 2 separately;
	\item\emph{Step 2}: Compute $\hat{\mu}^{n+1}$ and $\breve{\mu}^{n+1}$ by \eqref{muhat} and \eqref{mubre};
	\item\emph{Step 3}: Solve $\hat{u}^{n+1}$ and $\breve{u}^{n+1}$ from ACSAV Subproblem 3 and 4 separately;
	\item \emph{Step 4}: Calculate $V^{n+1}$ from \eqref{V^n+1};
	\item \emph{Step 5}: Calculate $\phi^{n+1}$ and $u^{n+1}$ from \eqref{Vn+1} ;
	\item \emph{Step 6}: Calculate $p^{n+1}$ from \eqref{ACNS-ACSAV4}.
\end{itemize}


\section{ACSAV method with explicit convection term for the ACNS model}\label{sec:nsacsav}

The nonlinear convection term in the NS equations can be made explicit in time discretization using the scalar auxiliary variable presented in Section \ref{sec:acsav}. We now present another version of the ACSAV method  that treats the convection term fully explicitly. This method is also unconditionally long time stable, second-order accurate, linear, and fully decoupled.

With $u|_{\partial \Omega}=0$ and $\nabla\cdot u=0$ we have
$$\int_\Omega (u\cdot\nabla)u\cdot u\,dx = \frac{1}{2}\int_{\partial\Omega}u\cdot \vec{n}|u|^2ds - \frac{1}{2}\int_{\Omega}(\nabla\cdot u)|u|^2\,dx = 0.$$
So the ACNS system is equivalent to
\begin{numcases}{}
\text{Equation} \eqref{eq:ACSAV-AC1}-\eqref{eq:ACSAV-AC3} \text{, Equation} \eqref{eq:ACSAV-NS2}, \nonumber\\
\partial_tu+ \frac{r(t)}{exp(-\frac{t}{T})}(u\cdot \nabla) u  -\nu \Delta u + \nabla p = \frac{r(t)}{exp(-\frac{t}{T})}\mu\nabla\phi,\label{eq:NSACSAV-NS1}\\
\frac{dr}{dt}=-\frac{1}{T} r + \frac{1}{exp(-\frac{t}{T})} \int_{\Omega} \Big((u \cdot \nabla\phi)\mu-(u \cdot \nabla\phi)\mu + (u\cdot\nabla)u\cdot u\Big)\, dx .\label{eq:NSACSAV-drdt}
\end{numcases}
The ACSAV method with explicit convection term then writes
\begin{Aalgorithm}\label{Algo-NS-ACSAV}
Given $u^{n-1}$, $u^n$, $p^{n-1}$, $p^n$, $\phi^{n-2}$, $\phi^{n-1}$, $\phi^n$, $q^{n-1}$, $q^n$, $r^{n-1}$, $r^n$, $\alpha >0$, $\beta > 0$, find $u^{n+1}$, $p^{n+1}$, $\phi^{n+1}$, $q^{n+1}$, and $r^{n+1}$ satisfying equation \eqref{ACNS-ACSAV1}-\eqref{ACNS-ACSAV2}, equation \eqref{mueq}, equation \eqref{ACNS-ACSAV4}, and
\begin{align}
&\frac{u^{n+1}-u^{n-1}}{2\Delta t}+ \frac{r^{n+1}+r^{n-1}}{2exp(-\frac{t^n}{T})}(u^{n}\cdot\nabla) u^n -\beta \Delta t^{-1}\nabla \left( \nabla \cdot u^{n+1} - \nabla \cdot u^{n-1}\right)\notag\\
&\qquad\qquad-\nu\Delta \left(\frac{u^{n+1}+u^{n-1}}{2}\right)+\nabla p^n =\frac{r^{n+1}+r^{n-1}}{2exp(-\frac{t^n}{T})} \mu^n\nabla\phi^n ,\label{nACNS3}\\
&\frac{r^{n+1}-r^{n-1}}{2\Delta t}
= -\frac{1}{T} \frac{r^{n+1}+r^{n-1}}{2}+ \frac{1}{exp(-\frac{t^{n}}{T})} \int_{\Omega}( u^{n} \cdot \nabla \phi^{n}) \bar{\mu}^n  \, dx ,\notag\\
&\qquad\qquad - \frac{1}{exp(-\frac{t^{n}}{T})} \int_{\Omega} (\frac{u^{n+1}+u^{n-1}}{2} \cdot \nabla \phi^{n}) \cdot\mu^n \, dx \notag\\
&\qquad\qquad + \frac{1}{exp(-\frac{t^{n}}{T})} \int_{\Omega} (u^{n} \cdot \nabla) u^{n} \cdot \frac{u^{n+1}+u^{n-1}}{2}  \, dx. \label{ndrdt1}
\end{align}
\end{Aalgorithm}

\begin{theorem}
[Stability of Algorithm \eqref{Algo-NS-ACSAV}] Taking $\alpha$ and $\beta$ such that $\alpha\beta\geq \frac{1}{4}$, then for any $N\geq 3$ the inequality \eqref{acsavineq} holds.
\end{theorem}
\begin{proof}
Following the steps in the proof of Theorem \ref{acsavthm} exactly, we obtain \eqref{acsavineq}.
\end{proof}

\subsection*{Implementation of the ACSAV method with explicit convection term}
The implementation of the ACSAV method with explicit convection term (ACSAV-ECT) is very similar to the original ACSAV algorithm.
We will solve ACSAV Subproblem 1 and 2 for $\hat{\phi}^{n+1}$, $\breve{\phi}^{n+1}$ respectively, and then solve the following two subproblems for $\hat{u}^{n+1}$, $\breve{u}^{n+1}$ respectively.
\begin{align}
&\frac{\hat{u}^{n+1}-u^{n-1}}{2\Delta t} -\beta \Delta t^{-1}\nabla \left( \nabla \cdot \hat{u}^{n+1} - \nabla \cdot u^{n-1}\right)\notag\\
&\qquad\qquad-\nu\Delta \left(\frac{\hat{u}^{n+1}+u^{n-1}}{2}\right)+\nabla p^n = 0,\tag{ACSAV-ECT Subproblem 3}\\
&\frac{\breve{u}^{n+1}}{2\Delta t}+ (u^{n}\cdot\nabla) u^n -\beta \Delta t^{-1}\nabla \left( \nabla \cdot \breve{u}^{n+1}\right)-\nu\Delta \left(\frac{\breve{u}^{n+1}}{2}\right) = \mu^n\nabla\phi^n . \tag{ACSAV-ECT Subproblem 4}
\end{align}

Plugging the expression \eqref{req} of $r^{n+1}$ into \eqref{ndrdt1}$\cdot \frac{r^{n+1}+r^{n-1}}{2}$ gives
\begin{align}
&(\frac{1}{\Delta t}+\frac{1}{T})exp(-\frac{2t^n}{T})(V^{n+1})^2-\frac{1}{\Delta t}exp(-\frac{t^n}{T})r^{n-1}V^{n+1}\\
& = V^{n+1} \int_{\Omega}( u^{n} \cdot \nabla \phi^{n}) \hat{\mu}^{n+1}  \, dx + (V^{n+1})^2 \int_{\Omega}( u^{n} \cdot \nabla \phi^{n}) \breve{\mu}^{n+1}  \, dx,\notag\\
&\quad - V^{n+1} \int_{\Omega} (\frac{\hat{u}^{n+1}+u^{n-1}}{2} \cdot \nabla \phi^{n}) \cdot\mu^n \, dx - (V^{n+1})^2 \int_{\Omega} (\frac{\breve{u}^{n+1}}{2} \cdot \nabla \phi^{n}) \cdot\mu^n \, dx \notag\\
&\quad + V^{n+1}\int_{\Omega} (u^{n} \cdot \nabla) u^{n} \cdot \frac{\hat{u}^{n+1}+u^{n-1}}{2}  \, dx + (V^{n+1})^2\int_{\Omega} (u^{n} \cdot \nabla) u^{n} \cdot \frac{\breve{u}^{n+1}}{2}  \, dx. \notag
\end{align}
Then we can compute $V^{n+1}$ by
\begin{align}\label{sub5}
A^{n+1}V^{n+1} + B^{n+1} = 0,
\end{align}
where
\begin{align}
& A^{n+1} = 
(\frac{1}{\Delta t}+\frac{1}{T})exp(-\frac{2t^{n}}{T}) 
 -\int_{\Omega}( u^{n} \cdot \nabla \phi^{n}) \breve{\mu}^{n+1}  \, dx
    +\int_{\Omega} (\frac{\breve{u}^{n+1}}{2} \cdot \nabla \phi^{n}) \cdot\mu^n \, dx\notag\\
&\qquad    -\int_{\Omega} (u^{n} \cdot \nabla) u^{n} \cdot \frac{\breve{u}^{n+1}}{2}  \, dx,
\label{nAn+1}\\
& B^{n+1} =  -\frac{r^{n-1}}{\Delta t} exp(-\frac{t^{n}}{T}) -
\int_{\Omega}( u^{n} \cdot \nabla \phi^{n}) \hat{\mu}^{n+1}  \, dx
+ \int_{\Omega} (\frac{\hat{u}^{n+1}+u^{n-1}}{2} \cdot \nabla \phi^{n}) \cdot\mu^n \, dx \notag\\
&\qquad - \int_{\Omega} (u^{n} \cdot \nabla) u^{n} \cdot \frac{\hat{u}^{n+1}+u^{n-1}}{2}  \, dx .\
\end{align}

\begin{theorem}
The scalar equation \eqref{sub5} admits a unique solution.
\end{theorem}
\begin{proof}
We have \eqref{tmp3} already.
Then taking the inner product of ACSAV-ECT Subproblem 4 with $\frac{\breve{u}^{n+1}}{2}$, we get 
\begin{equation}\label{tmp40}
\begin{split}
&\int_{\Omega} (\frac{\breve{u}^{n+1}}{2} \cdot \nabla \phi^{n}) \cdot\mu^n \, dx -\int_{\Omega} (u^{n} \cdot \nabla) u^{n} \cdot \frac{\breve{u}^{n+1}}{2}  \, dx \\
&= \frac{1}{4\Delta t}\|\breve{u}^{n+1}\|^2 + \frac{\beta}{2\Delta t}\|\nabla\cdot\breve{u}^{n+1}\|^2 + \frac{\nu}{4}\|\nabla\breve{u}^{n+1}\|^2 \\
&\geq 0.
\end{split}
\end{equation}
From \eqref{tmp3}, \eqref{tmp40} and the expression of $A^{n+1}$ in \eqref{nAn+1}, we conclude that
$$A^{n+1} > 0.$$
Therefore the scalar equation \eqref{sub5} admits a unique solution.
\end{proof}

\section{Numerical Experiments}\label{sec:result}
In this section, we first verify the convergence order of the proposed CNLFAC scheme (Algorithm \ref{Algo-ACNS}), ACSAV scheme (Algorithm \ref{Algo-ACNS-ACSAV}), and the ACSAV-ECT scheme (Algorithm \ref{Algo-NS-ACSAV}) with  numerical examples. Then several numerical tests of the phase field fluid model are performed to verify the stability, effectiveness, and efficiency of these schemes.

Since the simulations by the ACSAV algorithm and the ACSAV-ECT algorithm are extremely close, in sec.~\ref{62}--\ref{64} we only report the results for two algorithms: CNLFAC and ACSAV. Computatinoal efficiency will be reported in sec.~\ref{65} for all three algorithms.

\subsection{Convergence test}
The studied phase field model does not have a natural forcing term, which makes it difficult to construct a solution to be used in a convergence test. Herein we verify the convergence rate of the algorithms by computing the Cauchy difference, as is done in \cite{Han17}. 
Specifically, we run the simulations on four successively refined meshes with mesh size $h=\frac{\sqrt{2}}{5\times 2^{l-1}}, l=1,..,4$. We also take a constant ratio of the time step and mesh size so that $\Delta t = \frac{0.05}{\sqrt{2}}h$. Then we compute the Cauchy difference of the solutions at two successive levels, namely $v^{n+1}_{l+1}-v^{n+1}_{l},$ at time $t^{n+1}=(n+1)\Delta t=0.1$, for $v= \phi, u_1, u_2, p$, $l=1,2,3,4$. 

The initial condition for the phase variable $\phi$ is set as
\begin{align*}
\phi^0=0.24 \cos(2\pi x)\cos(2\pi y) + 0.4 \cos(\pi x)\cos(3\pi y).
\end{align*}
For velocity and pressure, we solve the following Stokes problem with the Taylor-Hood element pair $P^2-P^1$ and use the solution as the initial condition $(u^0, p^0)$: 
\begin{equation*}\label{eq:Stokes}
\left\{\begin{aligned}
& -\nu \Delta u+\nabla p=\mu^0 \nabla \phi^0,\\
&\nabla \cdot u = 0,
\end{aligned}\right.
\end{equation*}
\noindent where $\mu^0=\lambda (-\Delta \phi^0  + f(\phi^0))$. Note that the initial condition $(u^0, p^0)$ obtained from solving this problem satisfies the discrete inf-sup condition, which guarantees stability of the simulation in the first step.

The proposed CNLFAC algorithm is 3-level, the ACSAV algorithm and ACSAV-ECT algorithm are 4-level, so a 2-level method is needed to initialize the first two or three steps in time. We use the following first order, implicit algorithm to carry out this computation \cite{Feng06}.

\begin{Aalgorithm}[2-level, first order scheme for ACNS]\label{Algoo0}
Given $u^n$ and $\phi^n$, find $u^{n+1}$, $p^{n+1}$, and $\phi^{n+1}$ satisfying
\begin{align}
&\dot{\phi}^{n+1}-\lambda M \left( \Delta \phi^{n+1}-f_0(\phi^{n+1}, \phi^{n})\right)=0,\label{phio1}\\
&\frac{u^{n+1}-u^{n}}{\Delta t}+ u^{n}\cdot\nabla u^{n+1}-\nu \Delta u^{n+1}+\nabla p^{n+1} +\frac{1}{M}\dot{\phi}^{n+1}\cdot \nabla \phi^{n}=0,\label{uo1}\\
&\nabla \cdot u^{n+1}=0,\label{po1}
\end{align}
where 
\begin{align*}
&\dot{\phi}^{n+1}=\frac{\phi^{n+1}-\phi^{n}}{\Delta t}+ u^{n+1}\cdot \nabla \phi^{n}, \\
&f_0(\phi^{n+1}, \phi^{n})=\frac{1}{\eta^2}\frac{(\phi^{n+1})^2+(\phi^n)^2}{2}\frac{\phi^{n+1}+\phi^n}{2}-\frac{1}{\eta^2}\phi^{n}.
\end{align*}
\end{Aalgorithm}
Algorithm \ref{Algoo0} contains implicit nonlinear term $f_0(\phi^{n+1}, \phi^{n})$, which only appears in the first equation. So we use a Picard iteration algorithm to decouple the computation of $\phi$ and $u$ and apply Newton's method to the first nonlinear equation for the update of $\phi$. This procedure is listed in Algorithm \ref{picard}.
\begin{algorithm}[tp]
\caption{Picard iteration for Algorithm \ref{Algoo0}}\label{picard}
\begin{algorithmic}[1]
\Procedure{(Given $u^n, \phi^n$, Find $u^{n+1}$, $p^{n+1}$ and $\phi^{n+1}$)}{}
\State $u^{n+1} \gets 0$
\State $u^{n+1}_{temp} \gets u^n$
\While{$\Vert u^{n+1}-u^{n+1}_{temp}\Vert \geq $ tolerance}
\State Solve \eqref{phio1} with $u^{n+1} = u^{n+1}_{temp} $ using Newton's method and get $\phi^{n+1}$
\State Then solve \eqref{uo1}-\eqref{po1} with $\phi^{n+1}$ and get $u^{n+1}$
\State $u^{n+1}_{temp} \gets u^{n+1}$
\EndWhile
\EndProcedure
\end{algorithmic}
\end{algorithm}

\begin{table}[h!]
\centering
{\small
\caption{\noindent Cauchy differences  of numerical solutions computed by the CNLFAC scheme with inputs $\eta=0.1, M=10, \lambda=0.0001, \nu=0.8$.}
\label{tab:CNLFAC}
\begin{tabular}{|c||c|c||c|c||c|c|}
\hline
$l$& $\Vert \nabla \phi^{n+1}_{l+1}-\nabla \phi^{n+1}_{l}\Vert$ & rate&  $\Vert \nabla u^{n+1}_{l+1}-\nabla u^{n+1}_{l}\Vert$ & rate& $\Vert p^{n+1}_{l+1}-p^{n+1}_{l}\Vert$ & rate \\
\hline
1  & 3.7409e-01  &   --- &  2.5467e-05  &   --- & 5.4504e-05 &   ---  \\
2  &  9.7100e-02  &  1.95  &  6.7258e-06  &   1.92  & 1.2777e-05 &   2.09    \\
3  &  2.4751e-02  &   1.97  &  1.4333e-06  &   2.23   & 1.7790e-06 &   2.84   \\
4  &  6.2438e-03  &   1.99  &  2.1241e-07  &   2.75  & 3.6667e-07 &   2.28   \\
\hline
\end{tabular}
}
\end{table}

\begin{table}[h!]
\centering
{\small
\caption{\noindent Cauchy differences of numerical solutions computed by the ACSAV scheme with inputs $\eta=0.1, M=10, \lambda=0.0001, \nu=0.8$.}
\label{tab:ACSAV}
\begin{tabular}{|c||c|c||c|c||c|c|}
	\hline
	$l$&  $\Vert \nabla \phi^{n+1}_{l+1}-\nabla \phi^{n+1}_{l}\Vert$ & rate & $\Vert \nabla u^{n+1}_{l+1}-\nabla u^{n+1}_{l}\Vert$ & rate& $\Vert p^{n+1}_{l+1}-p^{n+1}_{l}\Vert$ & rate \\
	\hline
	1  &  3.7412e-01  &   --- &  2.9741e-05  &   --- & 2.4582e-05 &   ---   \\
	2  &  9.7102e-02  &  1.95  &  7.0300e-06  &   2.08  & 7.5678e-06 &   1.70   \\
	3  &  2.4751e-02  &   1.97  &  1.4496e-06  &   2.28  & 1.7154e-06 &   2.14   \\
	4  &  6.2438e-03  &   1.99  &  2.1416e-07  &   2.76  & 3.6739e-07 &   2.22   \\
	\hline
\end{tabular}
}
\end{table}

\begin{table}[h!]
\centering
{\small 
\caption{\noindent Cauchy differences  of numerical solutions computed by the ACSAV-ECT scheme with inputs $\eta=0.1, M=10, \lambda=0.0001, \nu=0.8$.}
\label{tab:expACSAV}
\begin{tabular}{|c||c|c||c|c||c|c|}
\hline
$l$& $\Vert \nabla \phi^{n+1}_{l+1}-\nabla \phi^{n+1}_{l}\Vert$ & rate&  $\Vert \nabla u^{n+1}_{l+1}-\nabla u^{n+1}_{l}\Vert$ & rate& $\Vert p^{n+1}_{l+1}-p^{n+1}_{l}\Vert$ & rate \\
\hline
1  & 3.7412e-01  &   --- &  2.9747e-05  &   --- & 2.4589e-05 &   ---  \\
2  &  9.7101e-02  &  1.95  &  7.0320e-06  &   2.08  & 7.5611e-06 &   1.70    \\
3  &  2.4751e-02  &   1.97  &  1.4501e-06  &   2.28   & 1.7144e-06 &   2.14   \\
4  &  6.2438e-03  &   1.99  &  2.1425e-07  &   2.76  & 3.6585e-07 &   2.23   \\
\hline
\end{tabular}
}
\end{table}

We now test the convergence rates of the CNLFAC, ACSAV, and ACSAV-ECT schemes. The parameters in this problem are $\eta=0.1, M=10, \lambda=0.0001, \nu=0.8$.
The stabilization parameters are set as  $$\alpha = 1, ~\beta = 0.25.$$ 
Here we use the P2 finite element space for variable $\phi$, and the P2-P1 finite element spaces for $u$ and $p$. 
From the consistency error, one would expect second order convergence rate for $\phi$ in $H^1$ semi-norm, $u$ in $H^1$ semi-norm, and $p$ in $L^2$ norm, since the numerical error is $O(h^2+\Delta t^2) = O(\Delta t^2)$.

The Cauchy differences of numerical solutions computed by the CNLFAC scheme are listed in Table \ref{tab:CNLFAC}, for the phase field $\phi$, the fluid velocity $u$, and the fluid pressure $p$, illustrating that the CNLFAC algorithm is second order in time convergent.
For the ACSAV scheme, the Cauchy differences are listed in Table \ref{tab:ACSAV}, also showing second order in time convergent. 
Very similar results for the ACSAV-ECT scheme are presented in Table \ref{tab:expACSAV}.
Note that for the pressure projection scheme, which requires artificial boundary conditions for the pressure, the rate of convergence for pressure $p$ is first order, whereas, the artificial compression method we employ here avoids the requirement of artificial boundary conditions for $p$.

\subsection{Stability on simulating spinodal decomposition}\label{62}
In this section, we perform several tests on simulating the spinodal decomposition of binary fluids by referring to \cite{HJ20} and \cite{xfYang20} and show that our schemes satisfy the energy law, which in turn illustrates the long-time stability of our schemes.

The initial condition for the velocity and pressure is zero in the computational domain $\Omega = [0,2\pi]^2$.
The phase field variable $\phi$ is initially treated as a random field
$$\phi^0=\bar{\phi}+r(x,y)$$
with a mean component $\bar{\phi}=0.0$ and random $r\in[-0.001,0.001]$. The parameter values are set as $\eta = 0.1$, $\lambda = 0.01$, $M = 100$, $\nu = 1.0$. 
Under the Allen--Cahn dynamics, the mixture undergoes a rapid phase separation, in which phases of the same composition rapidly aggregate, then develops to a slower coarsening process in which smaller droplets are graduately absorbed by larger ones.

In order to verify that the CNLFAC and ACSAV algorithms maintain energy stability without any time step condition, we plot the evolution chart of the total free energy \eqref{energy-q} computed with various time step sizes, as shown in Fig.~\ref{fig:energy}. 
In Fig.\ref{fig:CNLFACEnergy}, all the energy curves have a trend of monotonic decay, which confirms the unconditional stability of CNLFAC algorithm. Similarly, Fig.\ref{fig:ACSAVEnergy} shows the monotonic evolution of the energy for all time step sizes, so the unconditional stability of the ACSAV algorithm is illustrated. We set the time step to be $\Delta t=0.01$, we plots snapshots of the curves of $\phi$ at different times using both CNLFAC and ACSAV algorithms, respectively. In Fig.~\ref{fig:ACSAVphi} and Fig.~\ref{fig:CNLFACphi}, we observe that the final equilibrium solutions in both algorithm simulations are circular.

In addition, Fig.~\ref{fig:V_n+1} plots the evolutions of the auxiliary variable $V^{n+1}$ computed in the ACSAV algorithm with different time step sizes ranging from 0.01 to 0.00125. The magnitude of $V^{n+1}$ is always close to 1, which implies the ACSAV algorithm is stable and convergent.

\begin{figure}[tp]
	\begin{center}
		\subfigure[CNLFAC.]{\label{fig:CNLFACEnergy}
			\includegraphics[scale=0.55]{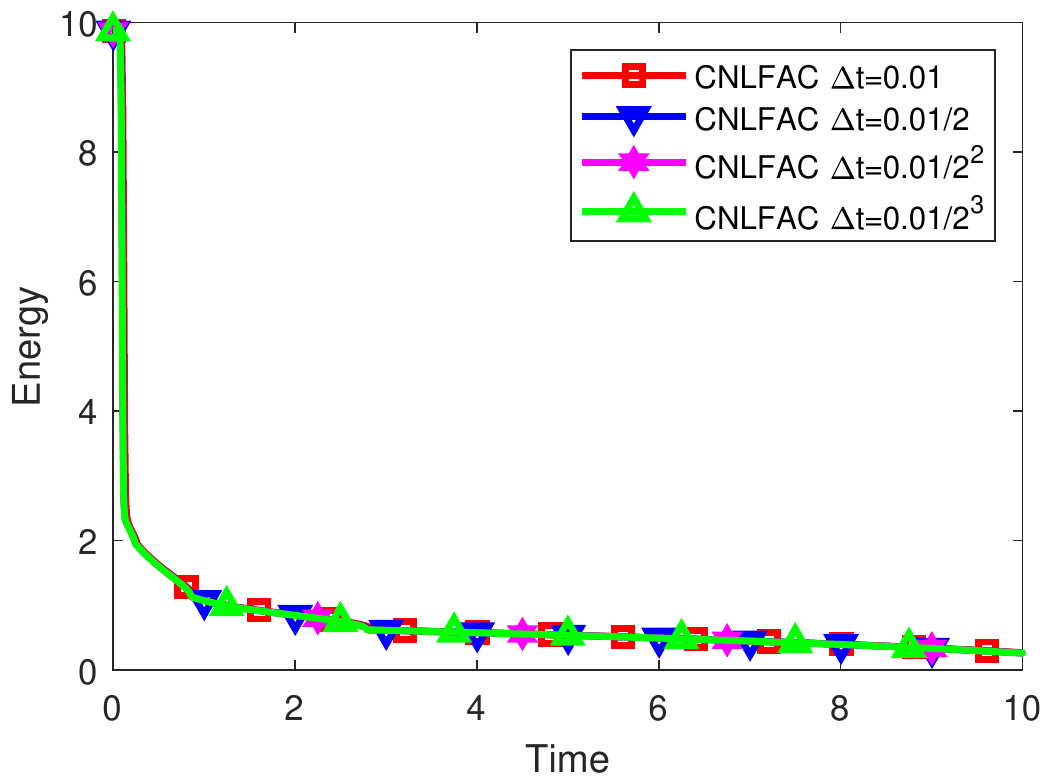}}
		\subfigure[AVSAV.]{\label{fig:ACSAVEnergy}
			\includegraphics[scale=0.55]{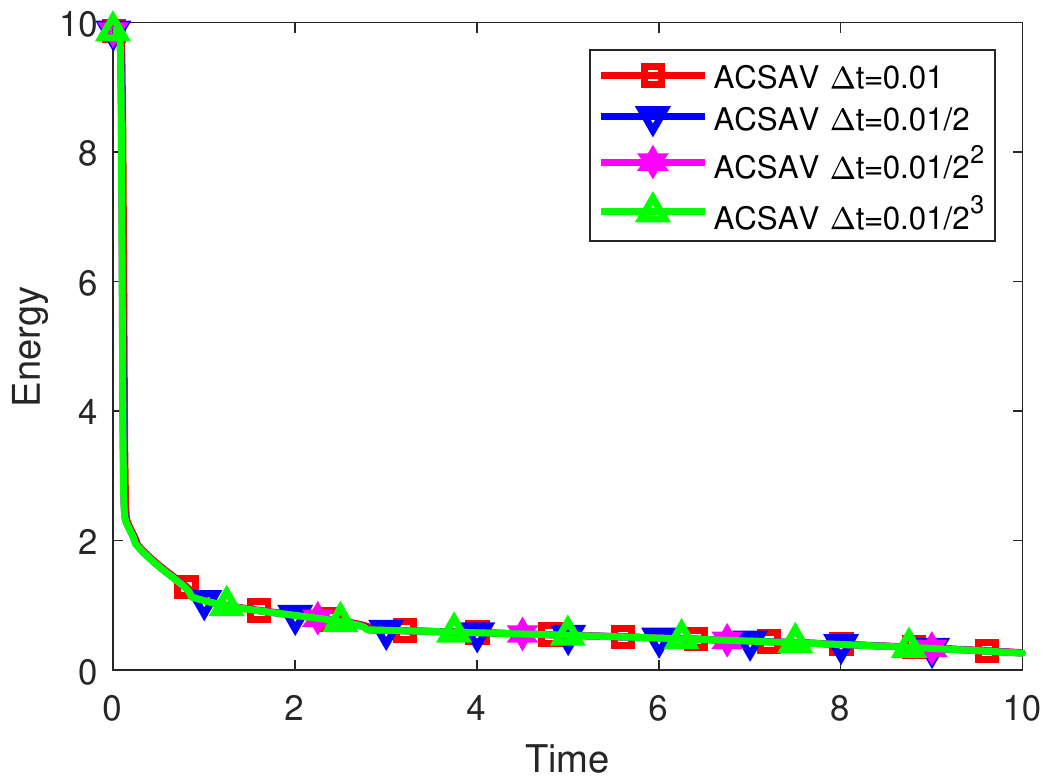}}
		\caption{Total free energy computed by the CNLFAC algorithm and the ACSAV algorithm. }
		\label{fig:energy}
	\end{center}
\end{figure}

\begin{figure}[tp]
	\begin{center}
		\subfigure[Snapshots are taken by CNLFAC algorithm at $t=0.15$, $t=0.3$, $t=0.75$, $t=1.0$, $t=2.0$ and $t=3.0$.]{\label{fig:CNLFACphi}
		\includegraphics[scale=0.25]{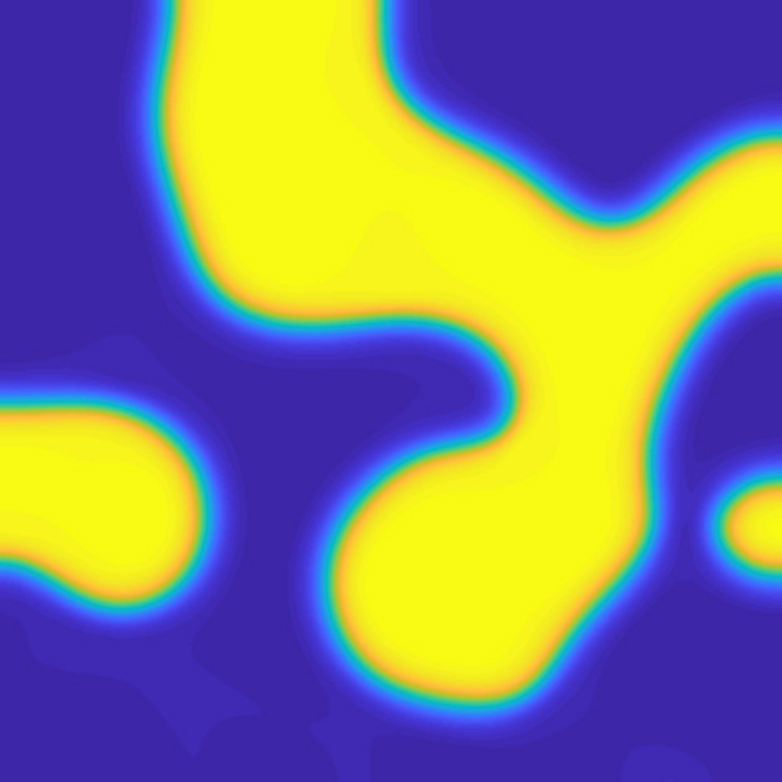} \includegraphics[scale=0.25]{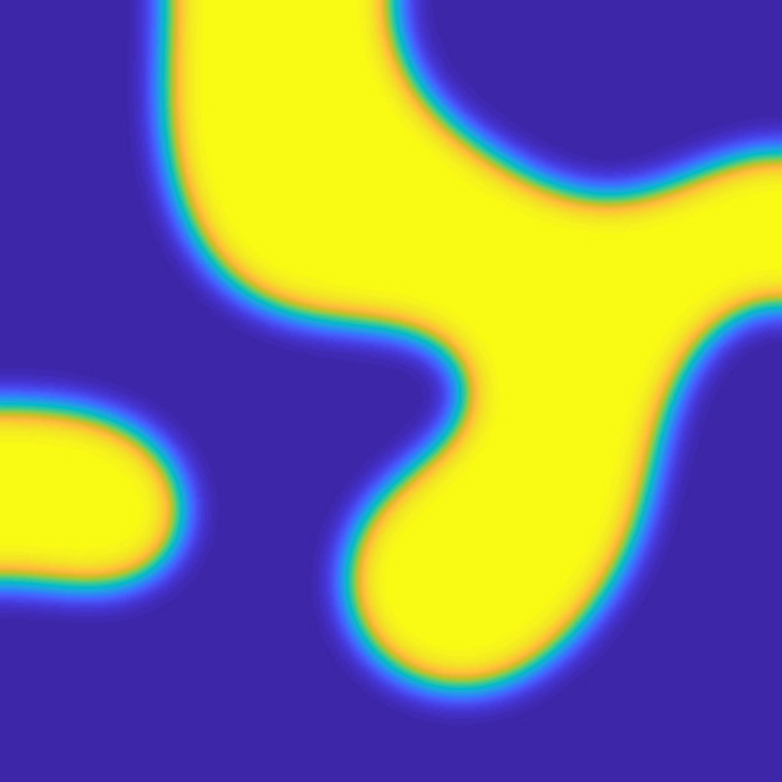}
		\includegraphics[scale=0.25]{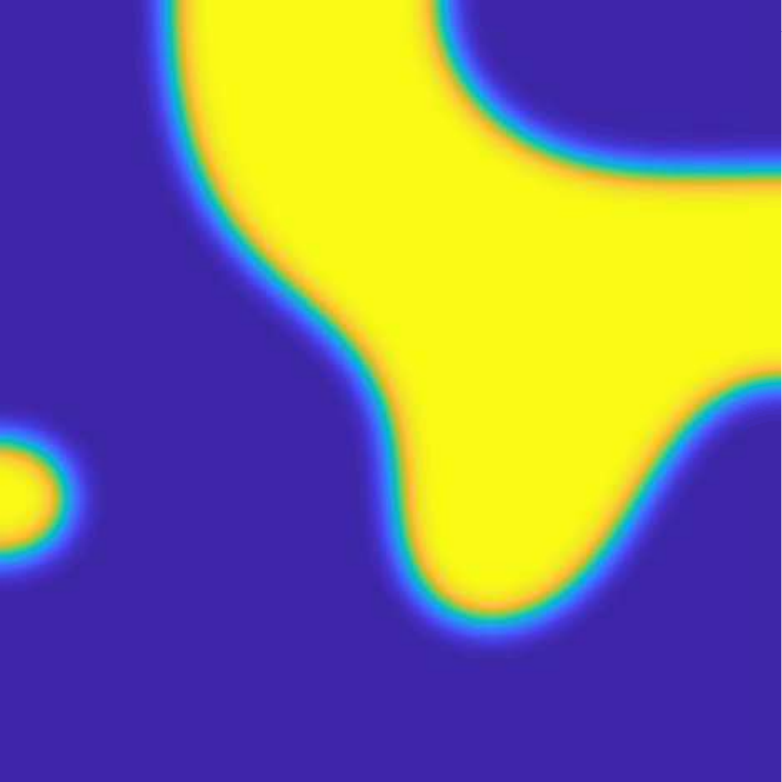}
		\includegraphics[scale=0.25]{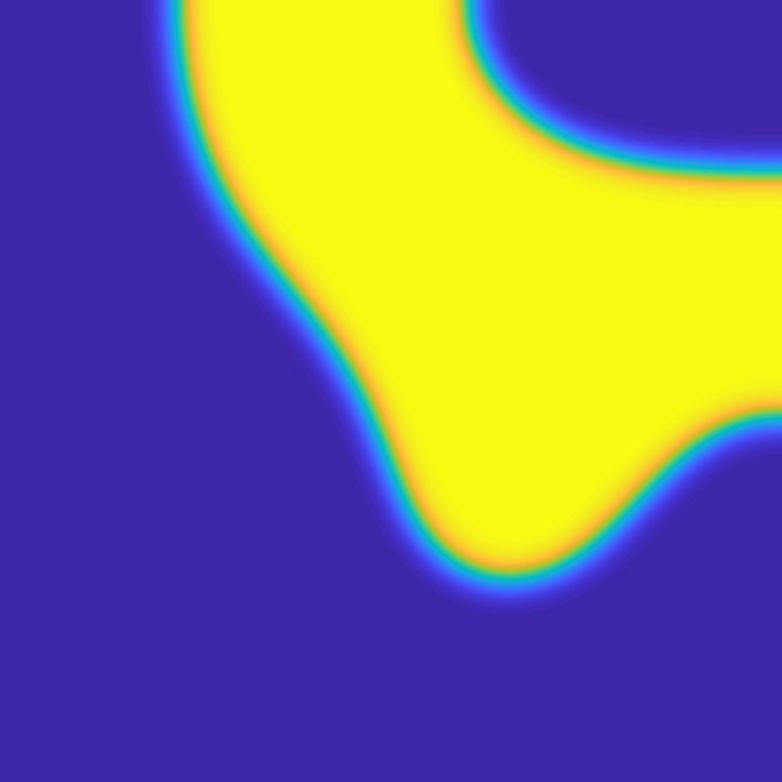}
		\includegraphics[scale=0.25]{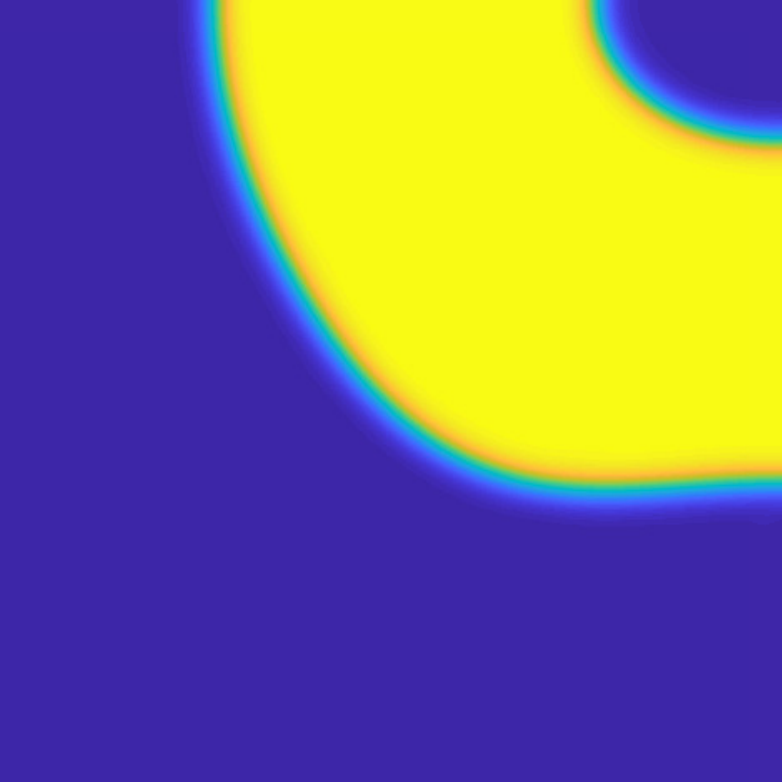}
		\includegraphics[scale=0.25]{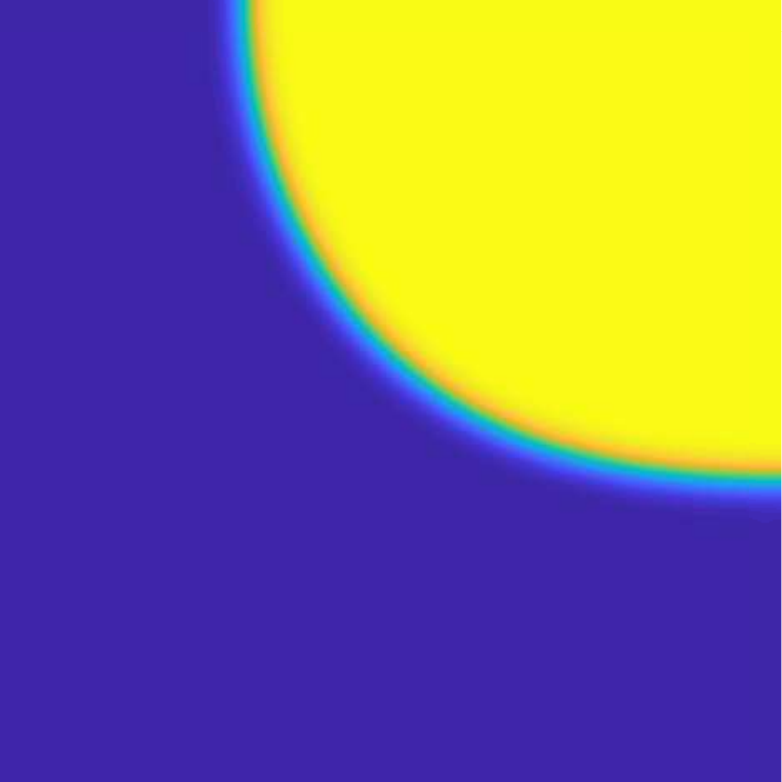}}
		\subfigure[Snapshots are taken by ACSAV algorithm at $t=0.15$, $t=0.3$, $t=0.75$, $t=1.0$, $t=2.0$ and $t=3.0$.]{\label{fig:ACSAVphi}
		\includegraphics[scale=0.25]{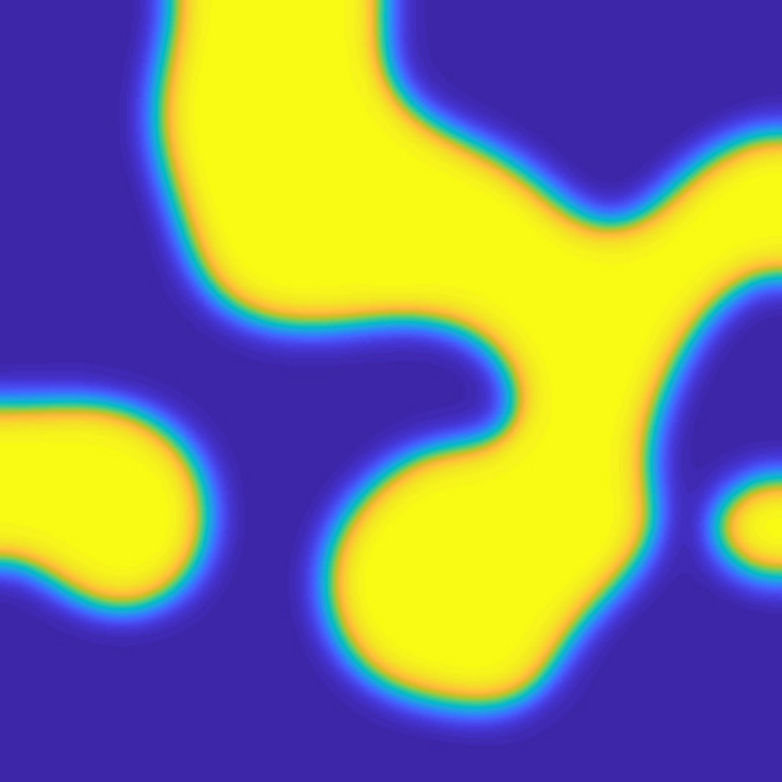} \includegraphics[scale=0.25]{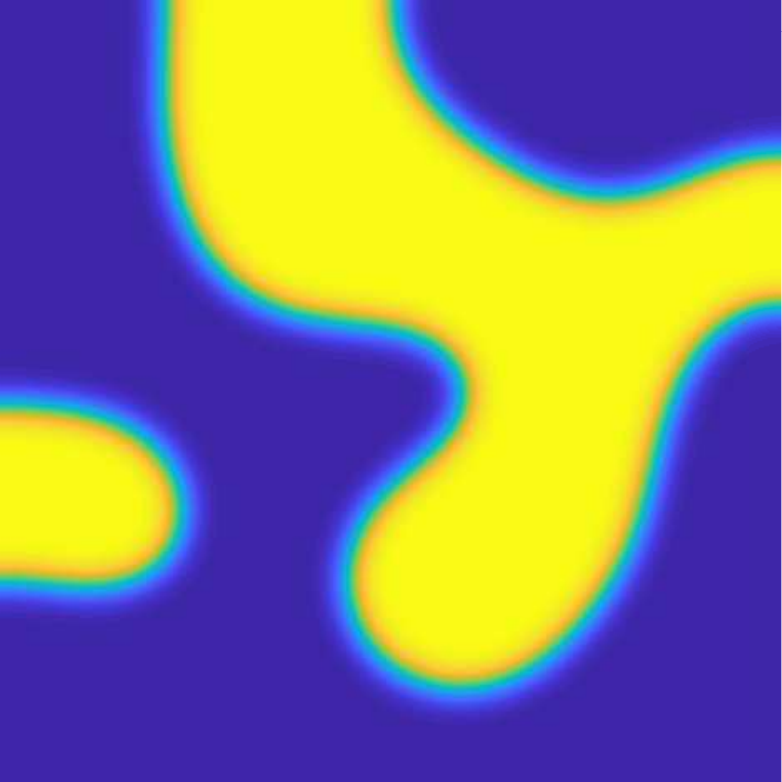}
		\includegraphics[scale=0.25]{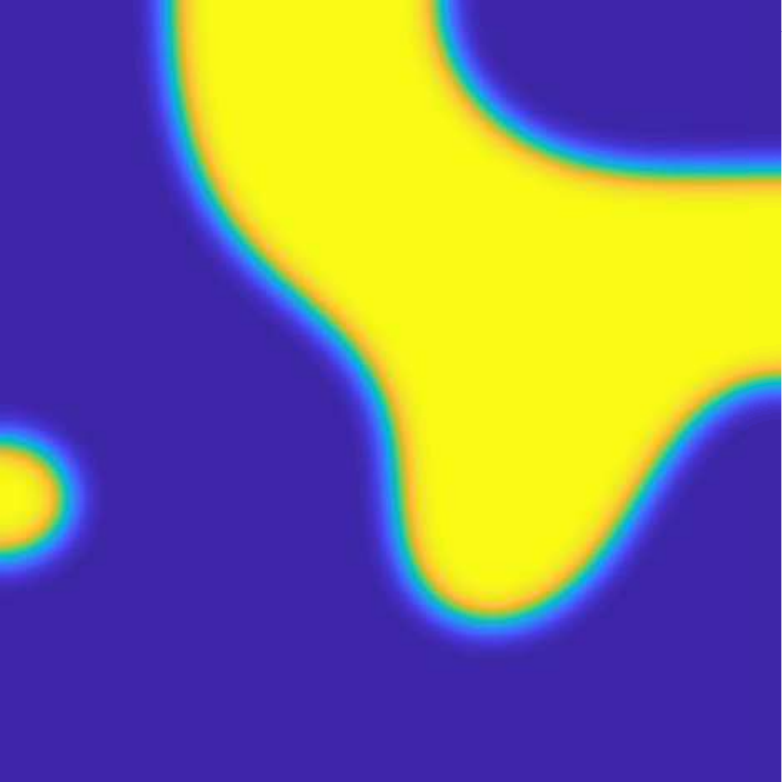}
		\includegraphics[scale=0.25]{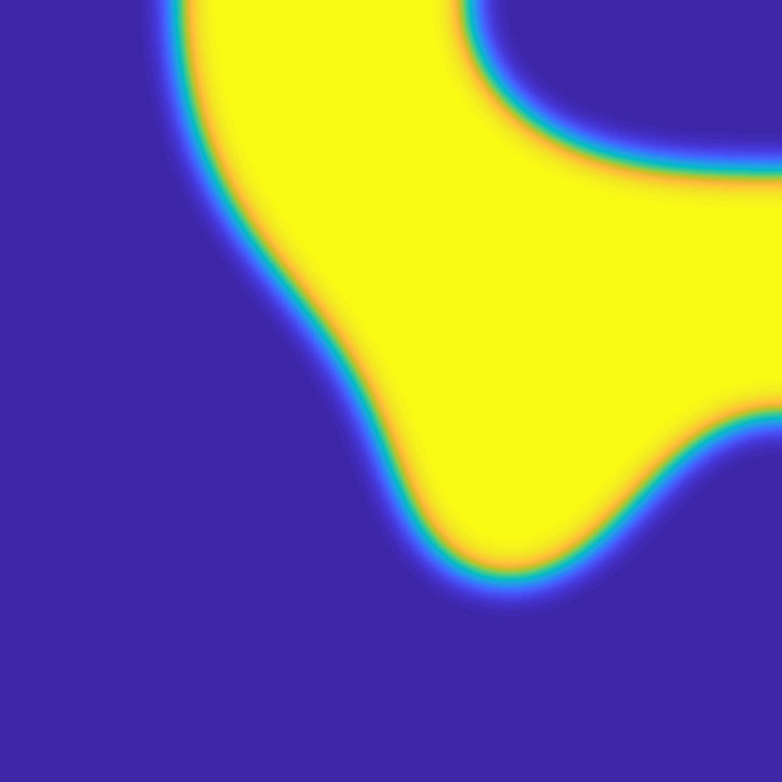}
		\includegraphics[scale=0.25]{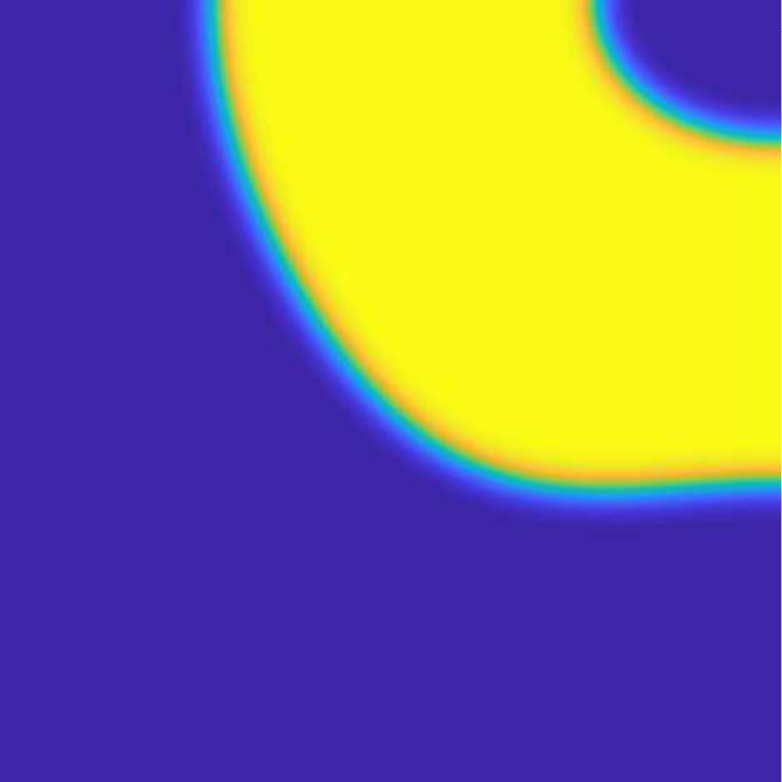}
		\includegraphics[scale=0.25]{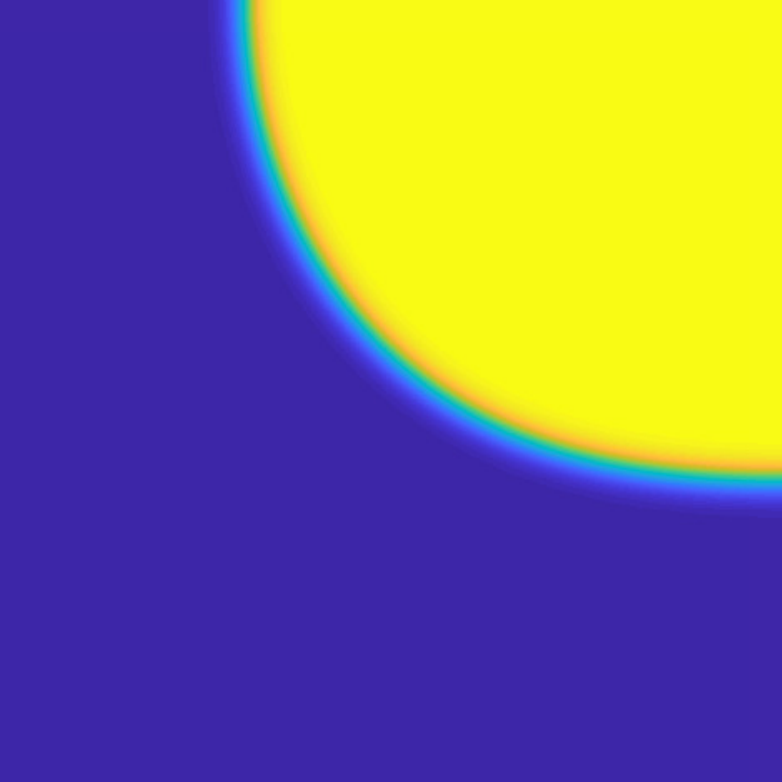}}
		\caption{The 2D dynamical evolution of the profile $\phi$ of the spin decomposition instances under the CNLFAV and ACSAV algorithms. }
	\end{center}
\end{figure}

\begin{figure}[tp]
	\begin{center}
		\includegraphics[scale=0.75]{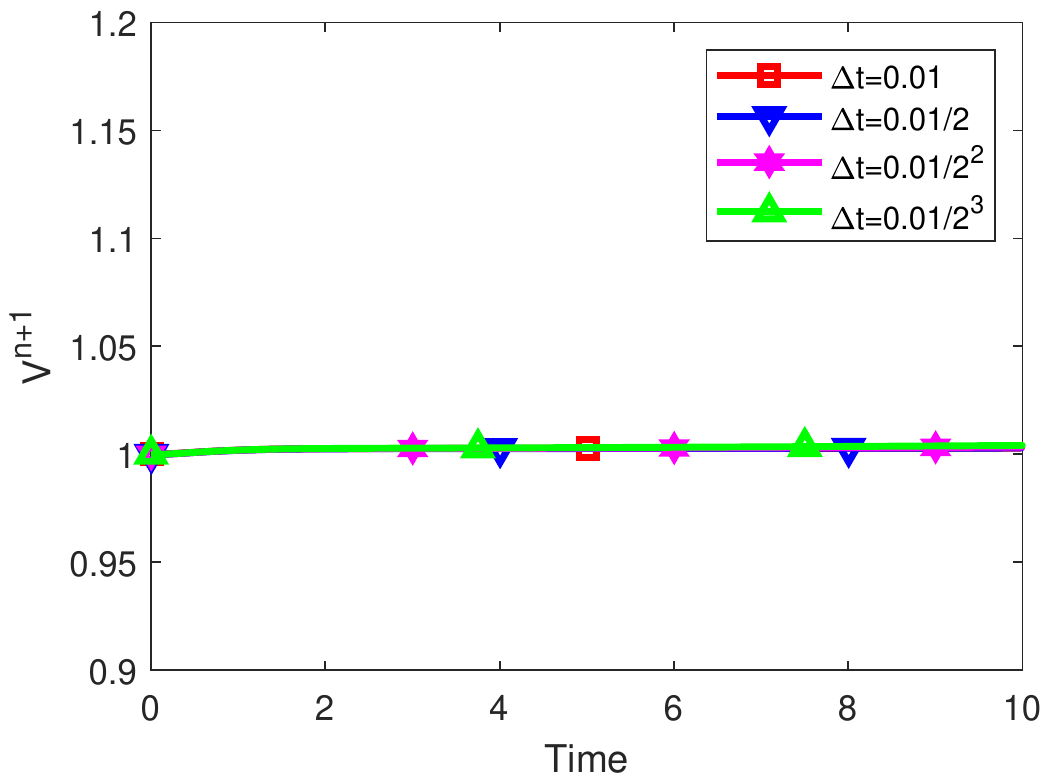}
		\caption{The evolution of $V^{n+1}$ computed in the ACSAV algorithm.}\label{fig:V_n+1}
	\end{center}
\end{figure}

\subsection{Shrinking circular bubble}
We refer to \cite{xfYang20} to simulate the shrinking process of a circular bubble in the computational domain $\Omega=[0,2\pi]^2$.
The initial conditions are given as follows:
\begin{equation*}
\left\{
\begin{aligned}
&\phi^0(x,y)=1+\sum_{i=1}^{2}\tanh(\frac{r_i-\sqrt{(x-x_i)^2+(y-y_i)^2}}{1.5\eta})\\
&{u}^0(x,y)={0},\;p^0(x,y)=0
\end{aligned}
\right.,
\end{equation*}
where $r_1=1.4$, $r_2=0.5$, $x_1=\pi-0.8$, $x_2=\pi+1.7$, $y_1=y_2=\pi$. The configuration profile of the initial condition of $\phi$ is given in the first pictures of Fig.~\ref{fig:AC-phi-evolution} and Fig.~\ref{fig:ACSAV-phi}. It shows two circles with different radii in the initial time step.
The parameters of the problem are  $\eta = 0.04$, $\lambda = 0.01$, $M = 10$, $\nu = 1$.

\begin{figure}[tp]
	\begin{center}
		\includegraphics[scale=0.23]{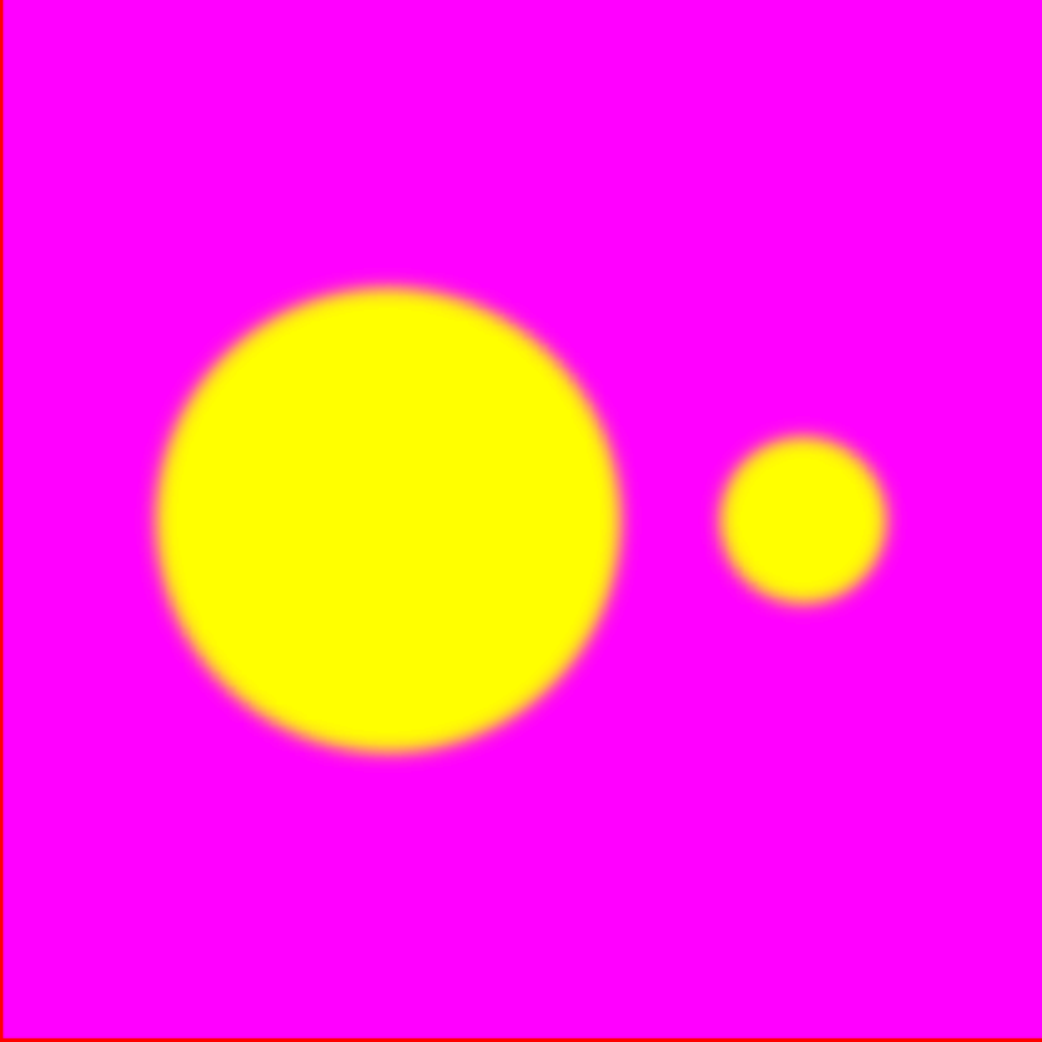} \includegraphics[scale=0.23]{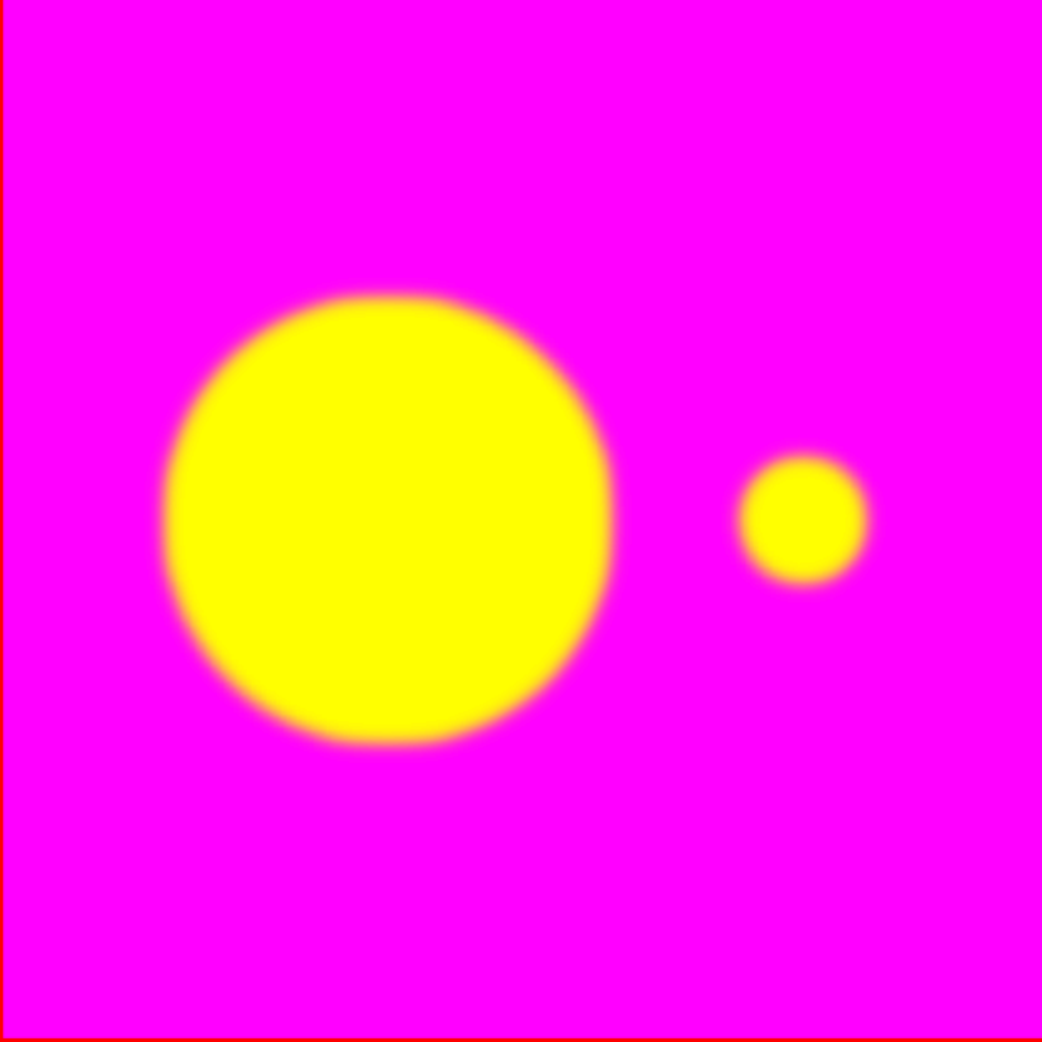}
		\includegraphics[scale=0.23]{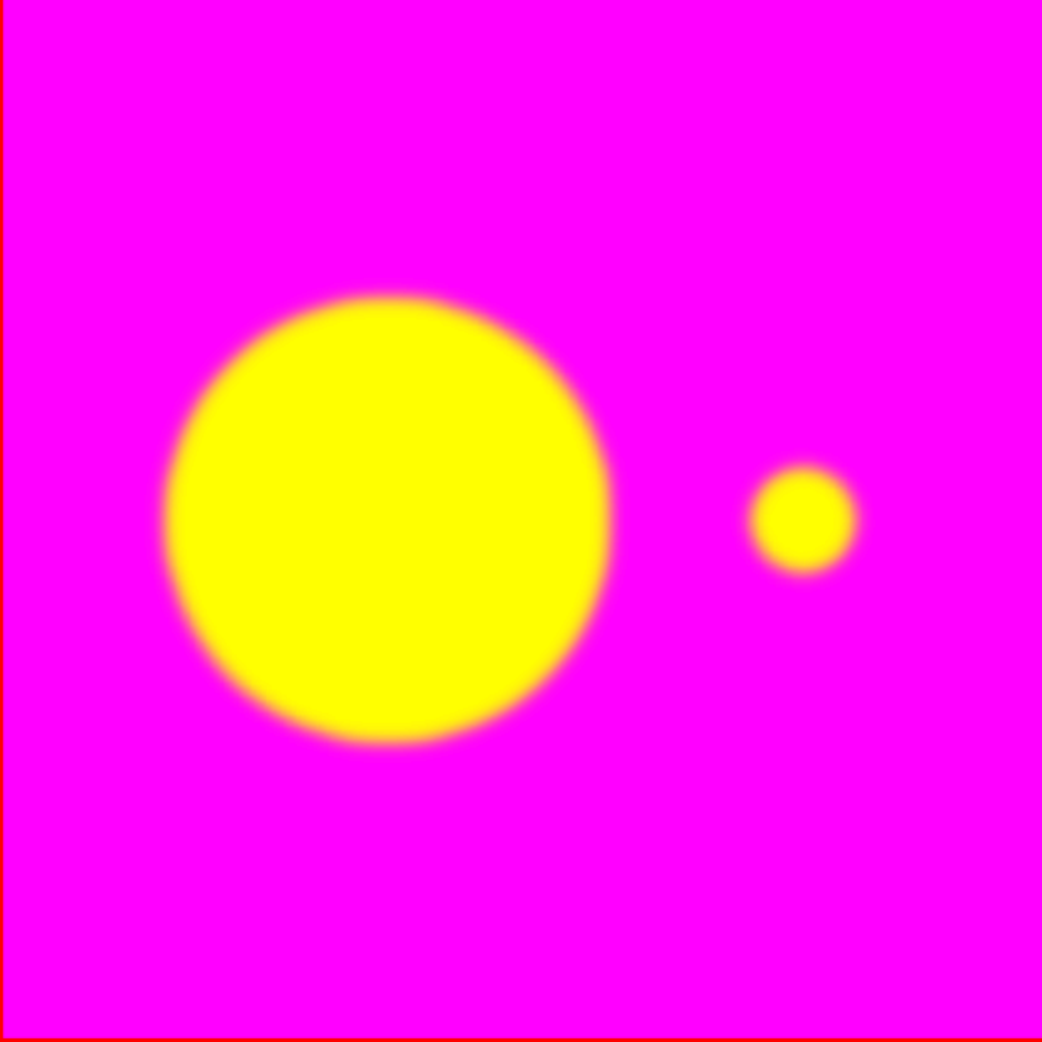}
		\includegraphics[scale=0.23]{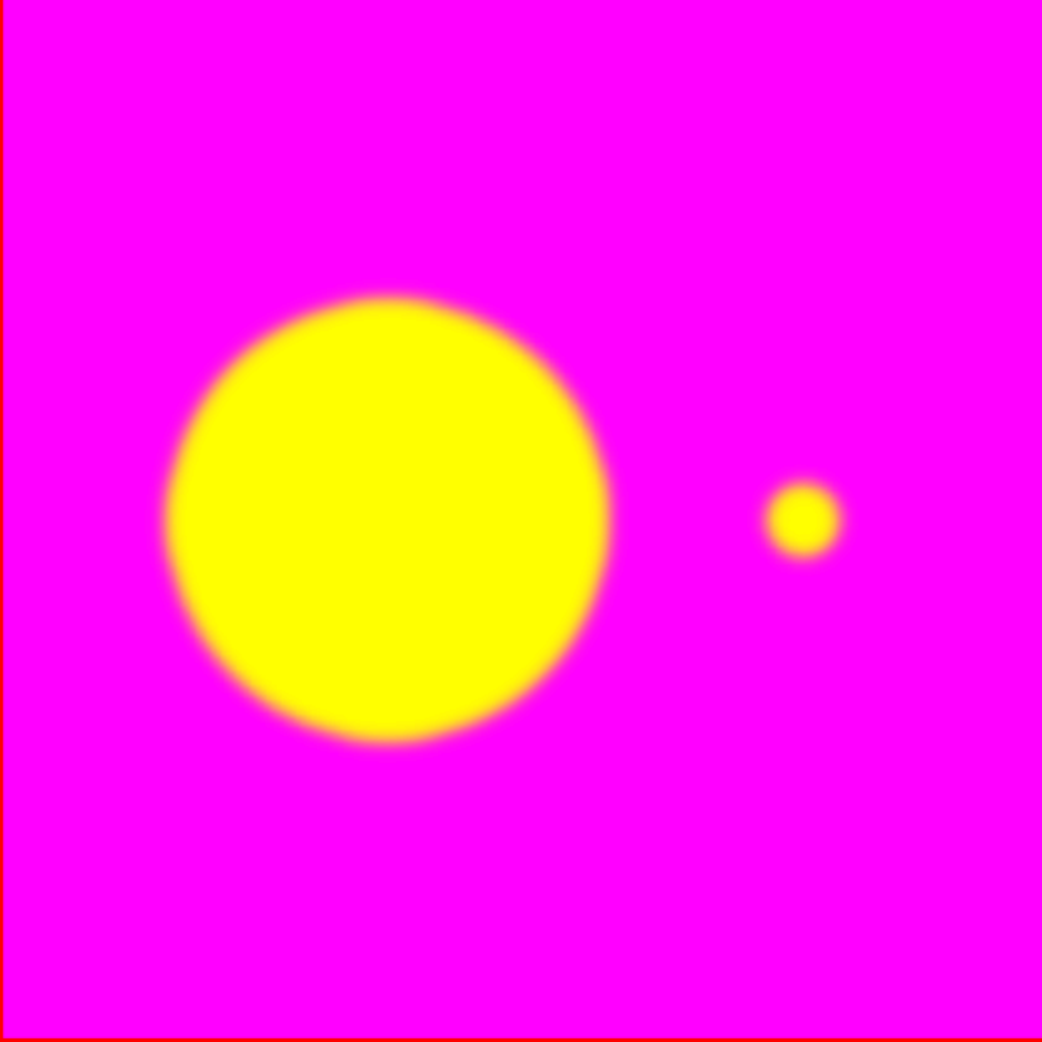}
		\includegraphics[scale=0.23]{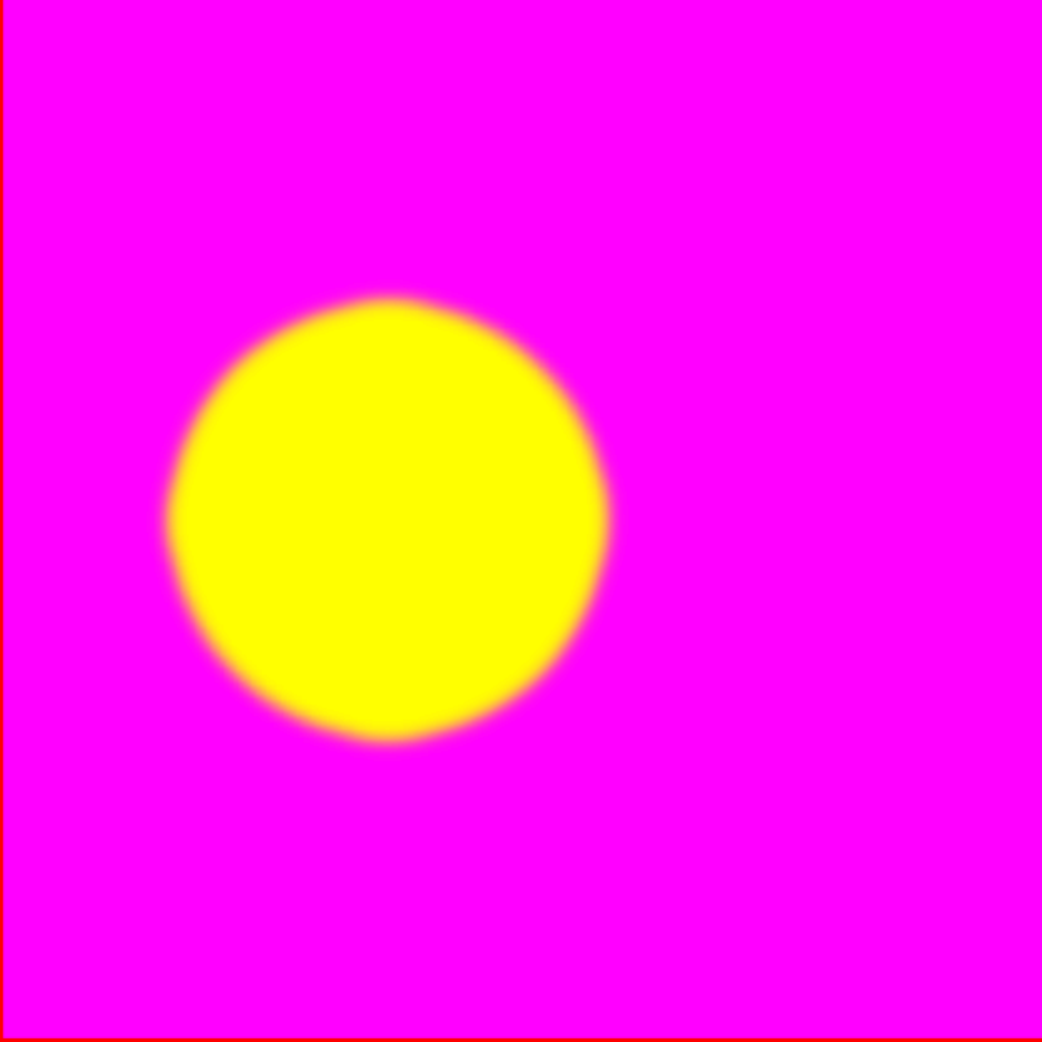}
		
		\caption{The variable $\phi$ computed by CNLFAC algorithm at $t=0$, $t=0.5$, $t=0.75$, $t=1.0$ and $t=1.25$. The parameters are $\eta = 0.04, \lambda = 0.01, M = 10, \nu = 1, \Delta t = 0.025, \Omega=[0,2\pi]\times[0,2\pi]$.}
		\label{fig:AC-phi-evolution}
	\end{center}
\end{figure}
\begin{figure}[tp]
	\begin{center}
		\includegraphics[scale=0.23]{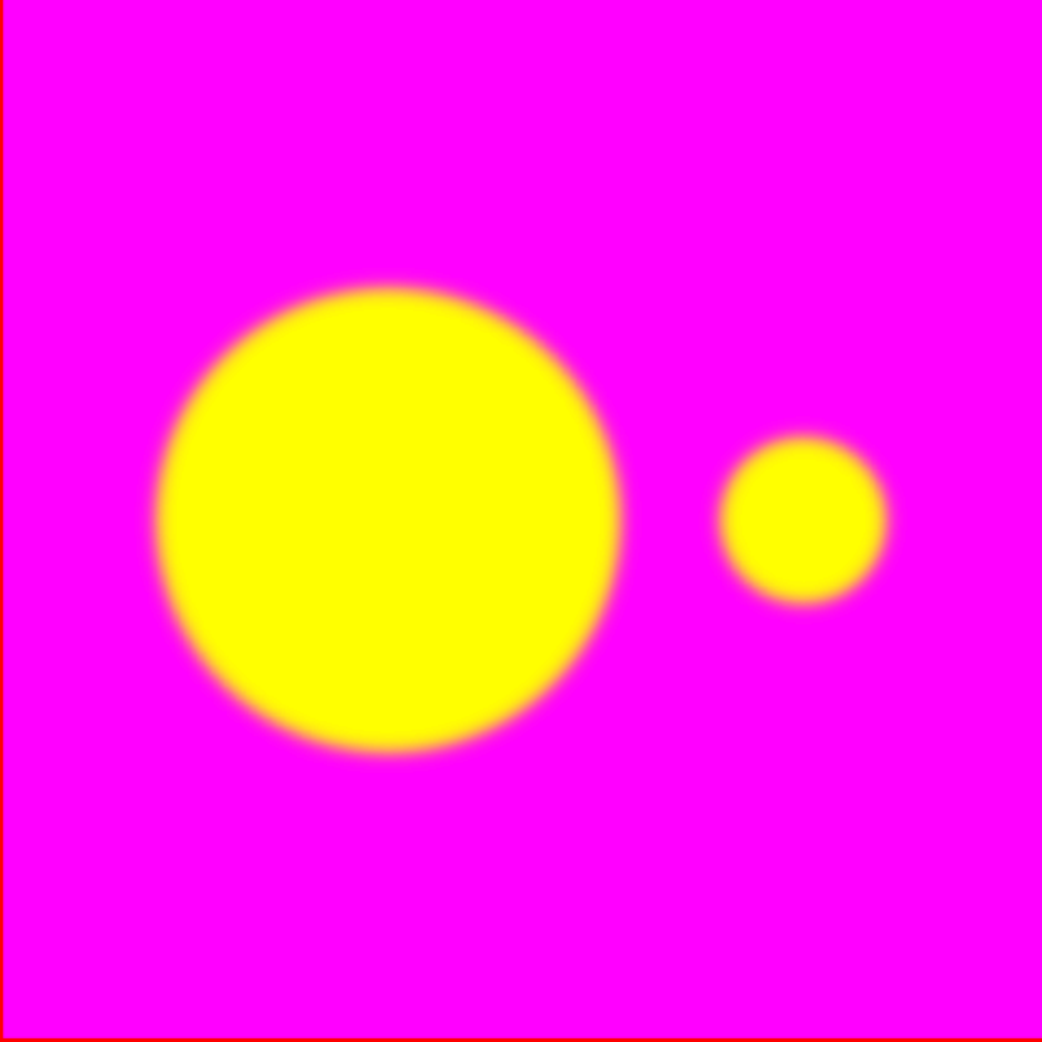} \includegraphics[scale=0.23]{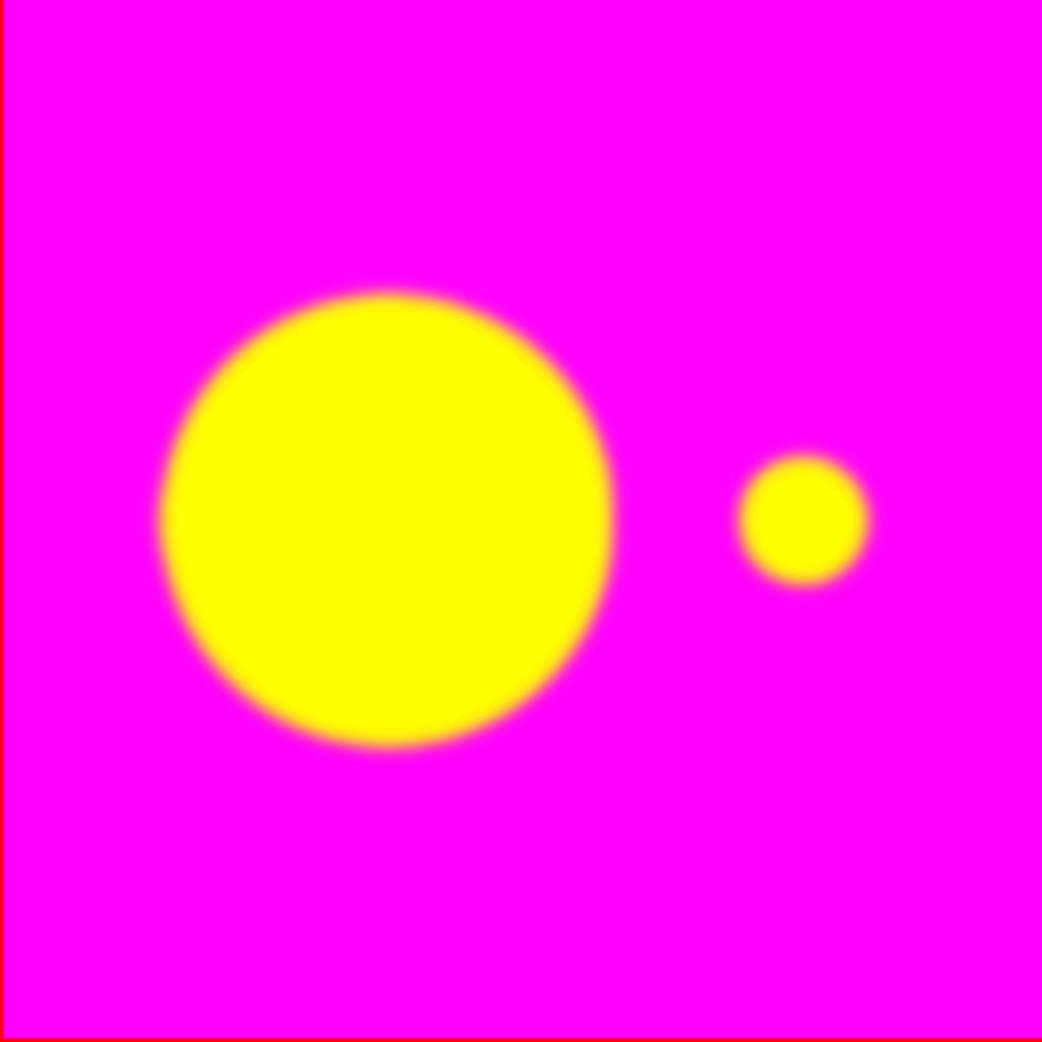}
		\includegraphics[scale=0.23]{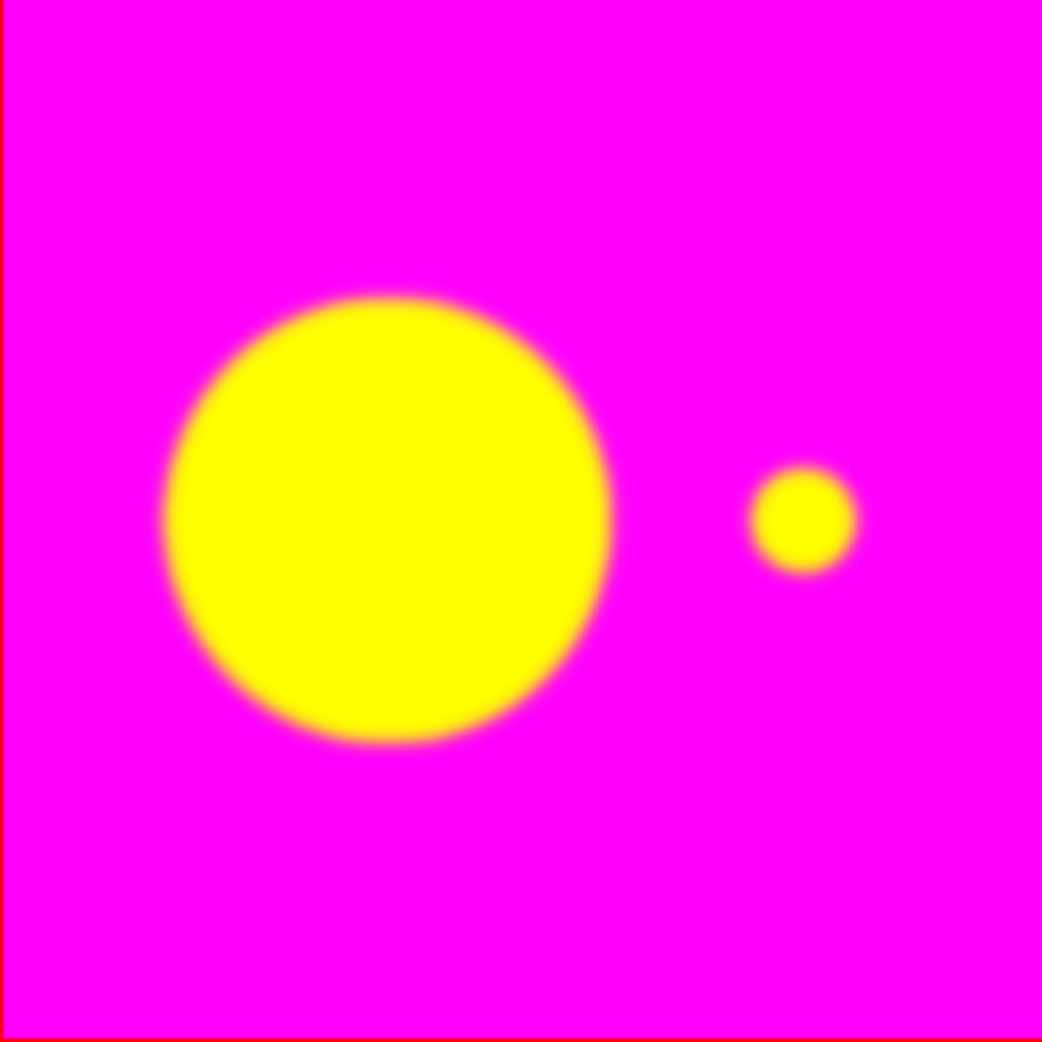}
		\includegraphics[scale=0.23]{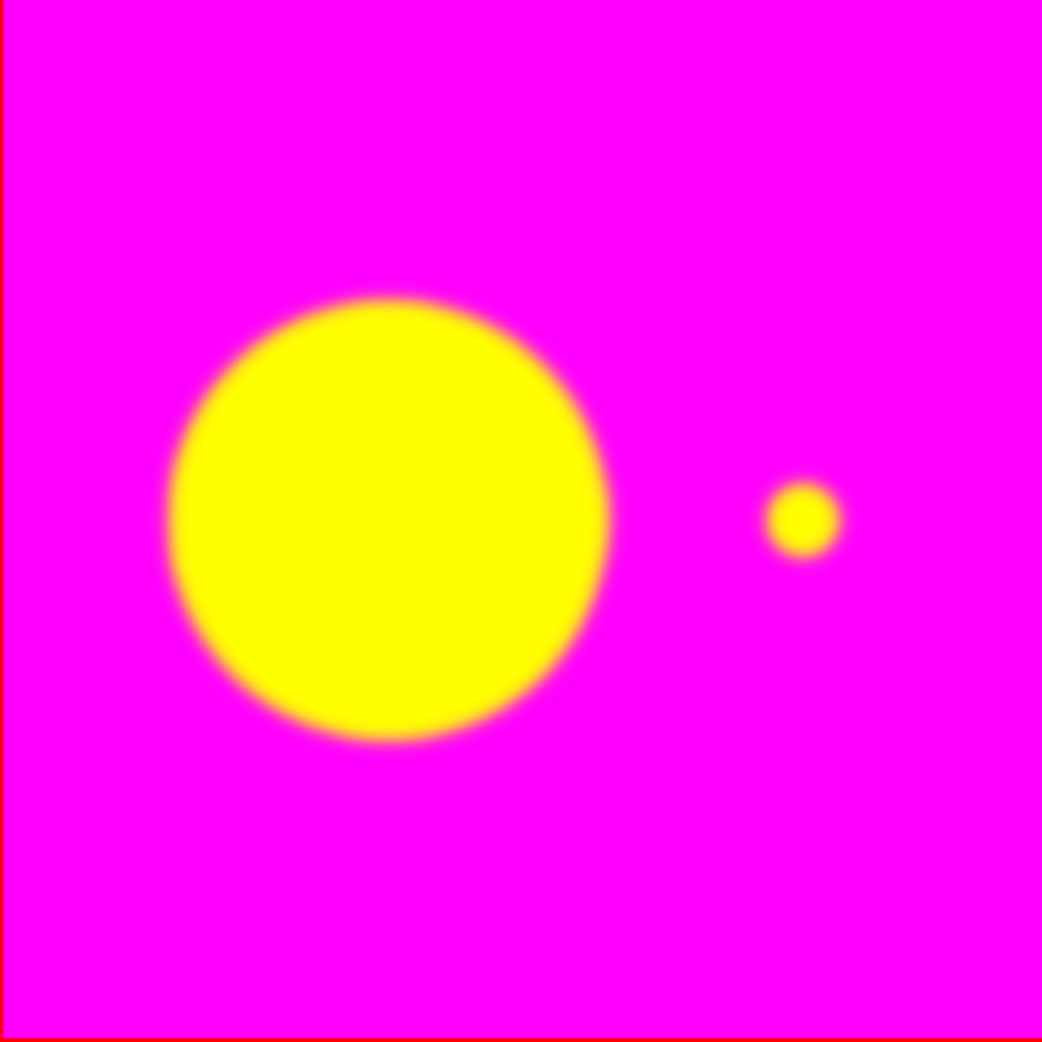}
		\includegraphics[scale=0.23]{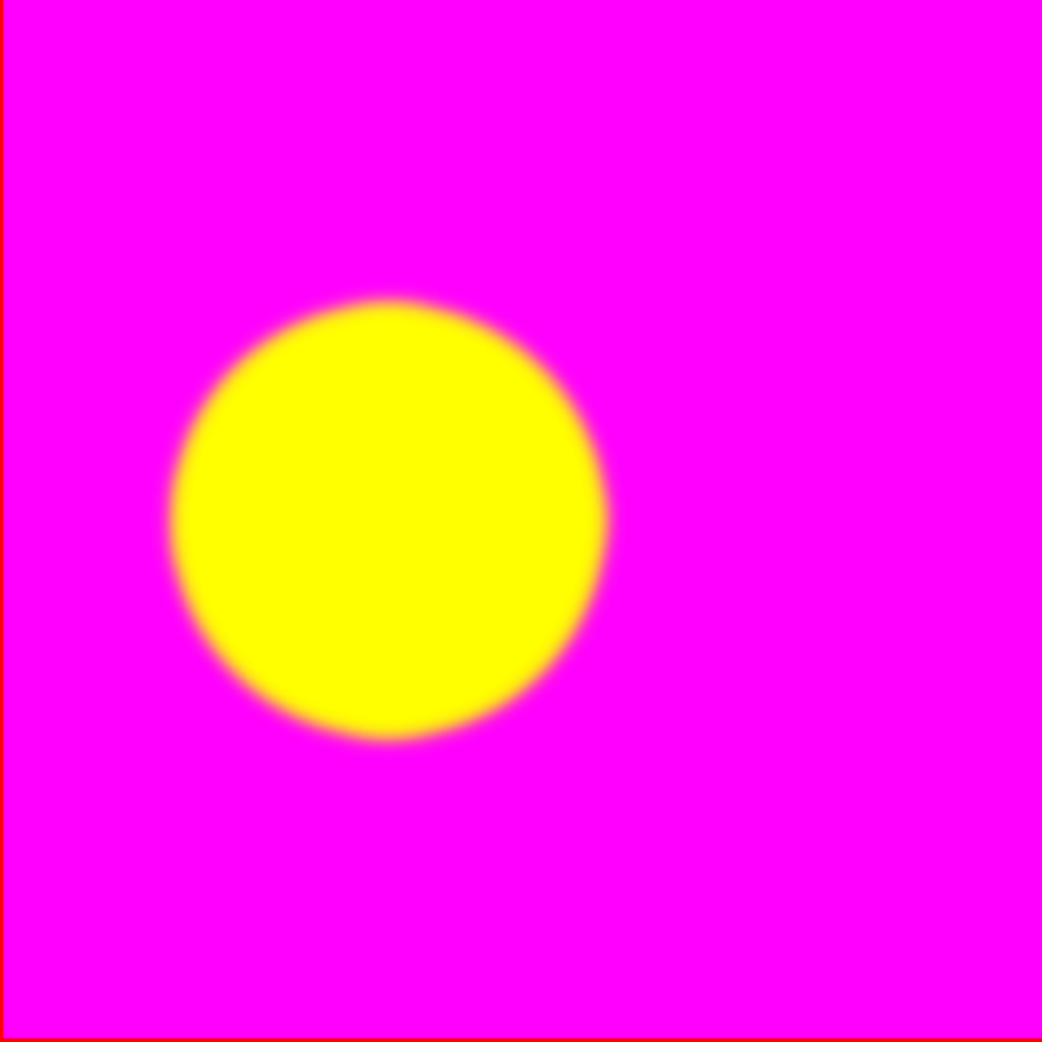}
		
		\caption{The variable $\phi$ computed by ACSAV algorithm at $t=0$, $t=0.5$, $t=0.75$, $t=1.0$ and $t=1.25$. The parameters are $\eta = 0.04, \lambda = 0.01, M = 10, \nu = 1, \Delta t = 0.025, \Omega=[0,2\pi]\times[0,2\pi]$.}
		\label{fig:ACSAV-phi}
	\end{center}
\end{figure}

We use CNLFAC and ACSAV algorithms to solve the ACNS system with time step size $\Delta t=0.025$. Fig.~\ref{fig:AC-phi-evolution} and Fig.~\ref{fig:ACSAV-phi} depict the profile state of the phase field variable $\phi$ at different times until steady state occurs. Due to the influence of roughening, the small circle is absorbed by the large circle, and the volume of the small circle gradually decreases. At $t=1.25$, the small circle disappears.

\subsection{Shape relaxation}\label{64}
In this section, we refer to \cite{Han17} to simulate the merging process of two circular bubbles using the ACNS system.

In the numerical experiment, the domain is set as $\Omega=[0,1.5]\times[0,1.5]$ and the initial conditions are given by
\begin{equation*}
\left\{
\begin{aligned}
&\phi^0(x,y)=1+\sum_{i=1}^{2}\tanh(\frac{r_i-\sqrt{(x-x_i)^2+(y-y_i)^2}}{\eta})\\
&{u}^0(x,y)={0},\;p^0(x,y)=0
\end{aligned}
\right.,
\end{equation*}
where $r_1=0.25$, $r_2=0.25$, $x_1=0.5$, $x_2=1$, $y_1=y_2=0.75$.
In this simulation, the parameters are $\eta = 0.02$, $\lambda = 0.01$, $M = 10$, $\nu = 1$, $\Delta t = 0.005$.

Fig.~\ref{fig:CNLFACkiss-phi} shows the merging process of two circular bubbles simulated by the CNLFAC algorithm. We find that, because of the surface tension, two adjacent circles quickly join together, gradually merge, and relax into a circle with the least surface energy. Fig.~\ref{fig:ACSAVkiss-phi} illustrates a similar process simulated by the ACSAV scheme. Apparently, simulations by the two schemes are almost identical.

\begin{figure}[tp]
	\centering
	\subfigure{
		\begin{minipage}[b]{.3\linewidth}
			\centering
			\includegraphics[scale=0.36]{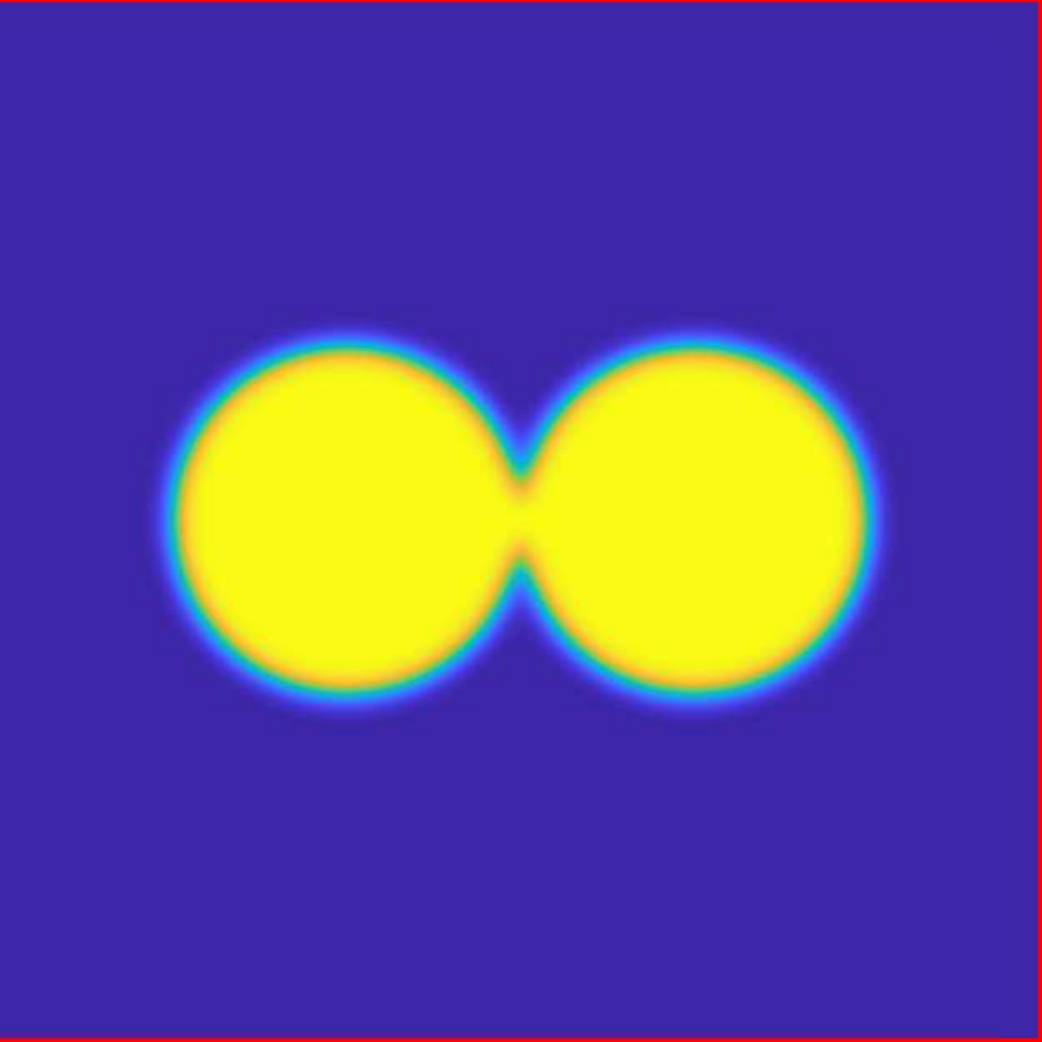}
		\end{minipage}
	}
	\subfigure{
		\begin{minipage}[b]{.3\linewidth}
			\centering
			\includegraphics[scale=0.36]{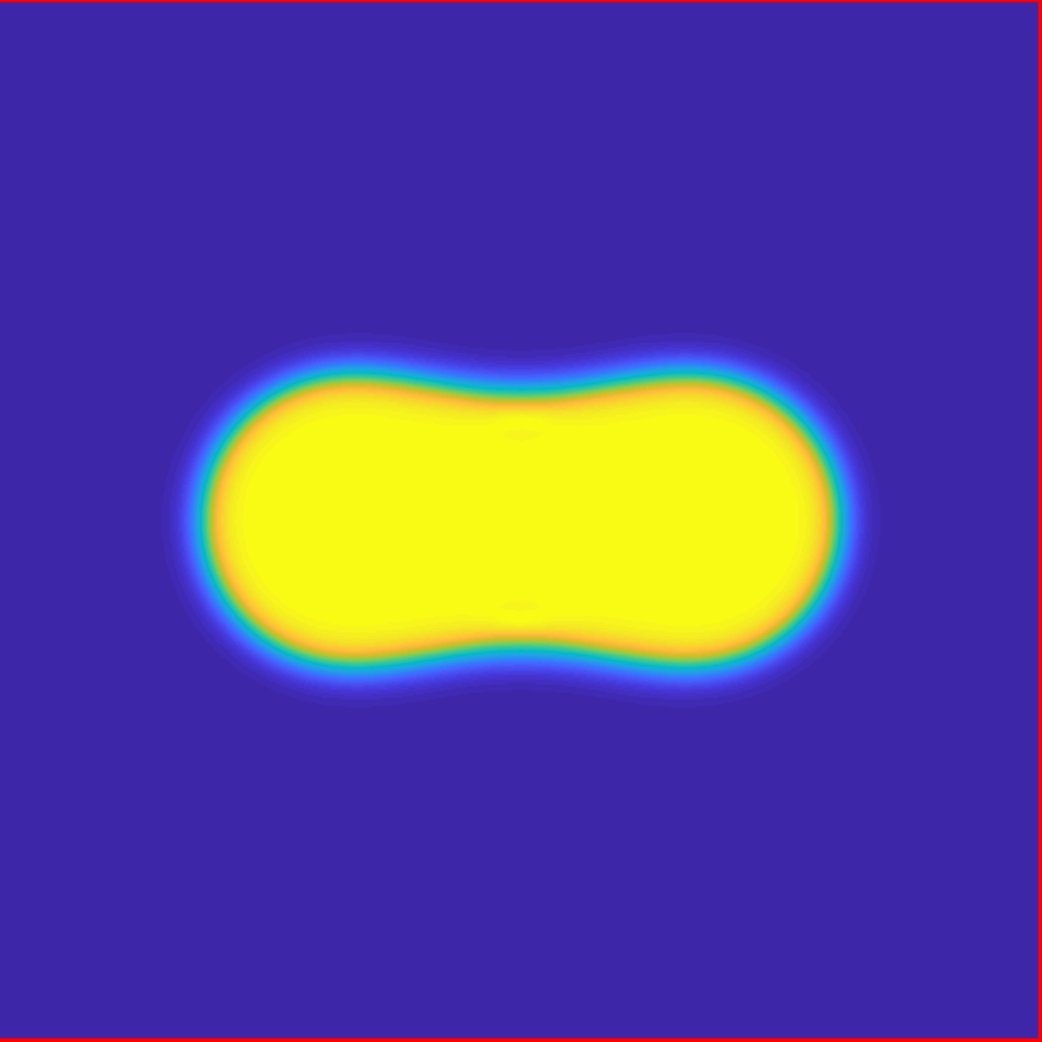}
		\end{minipage}
	}
	\subfigure{
		\begin{minipage}[b]{.3\linewidth}
			\centering
			\includegraphics[scale=0.36]{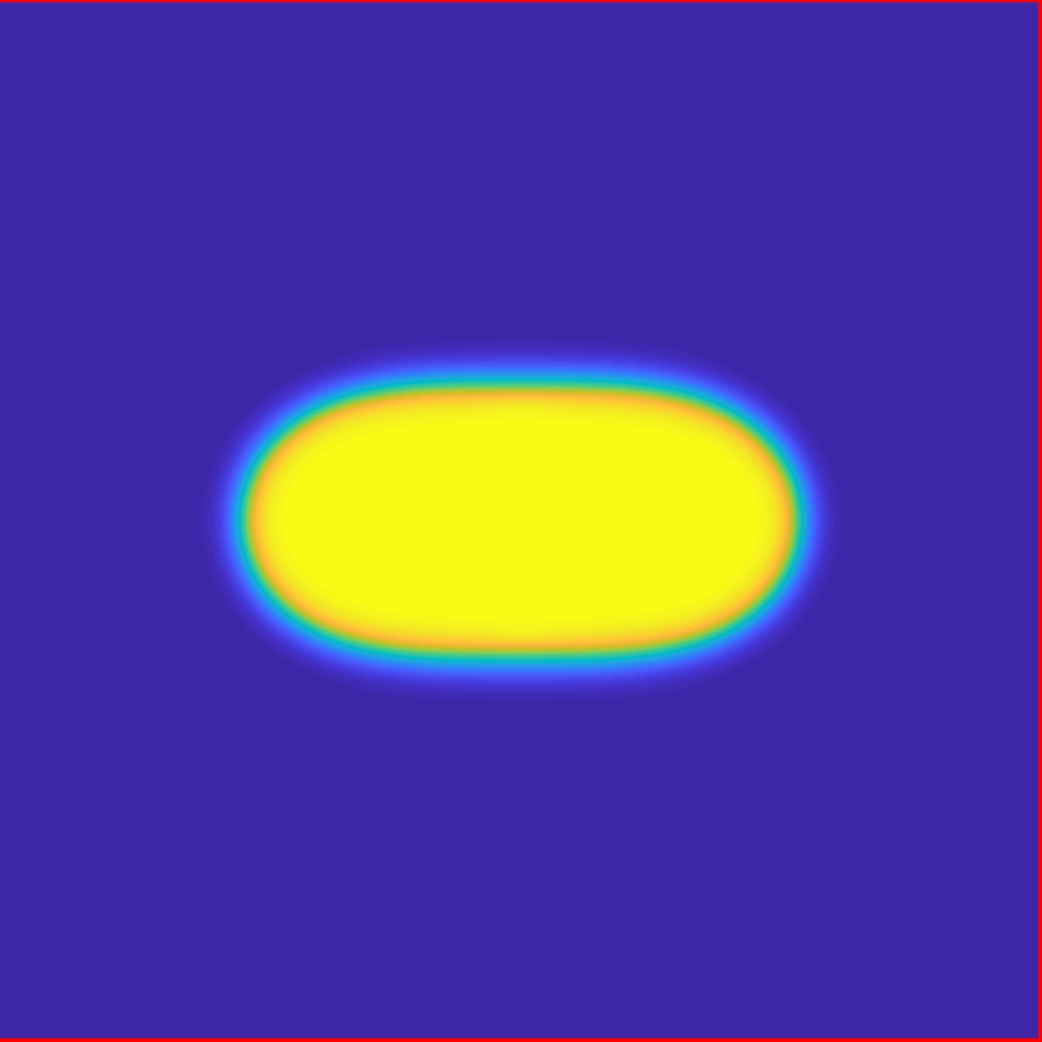}
		\end{minipage}
	}
	\subfigure{
		\begin{minipage}[b]{.3\linewidth}
			\centering
			\includegraphics[scale=0.36]{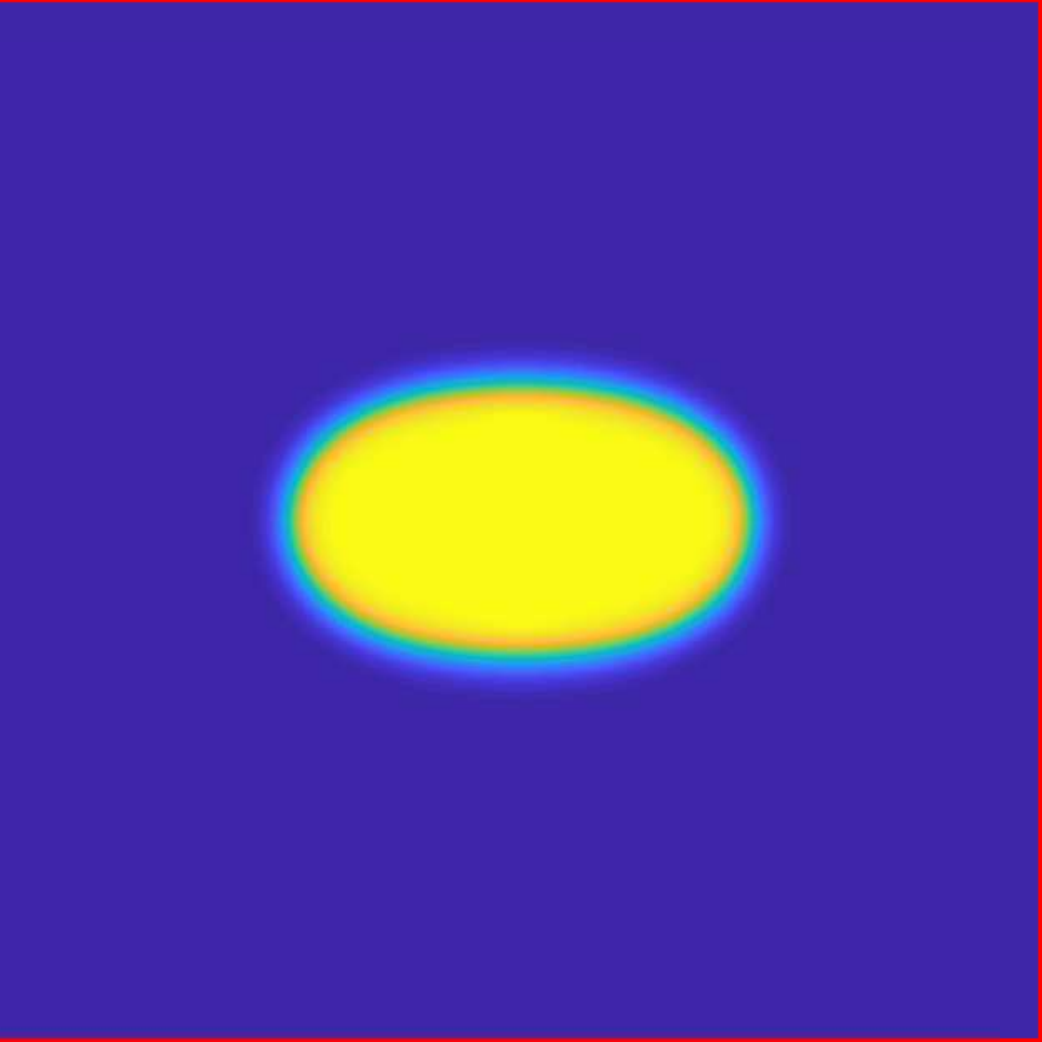}
		\end{minipage}
	}
	\subfigure{
		\begin{minipage}[b]{.3\linewidth}
			\centering
			\includegraphics[scale=0.36]{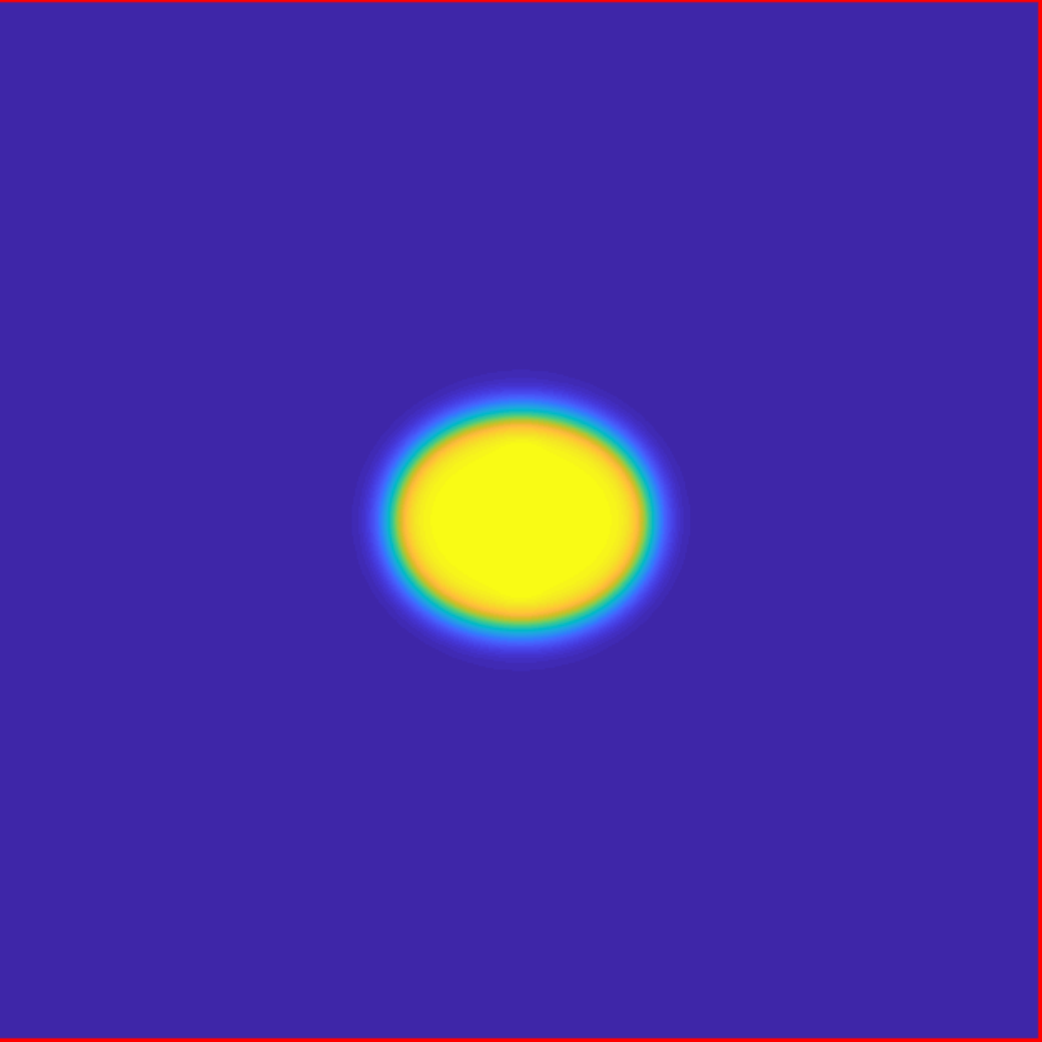}
		\end{minipage}
	}
	\subfigure{
		\begin{minipage}[b]{.3\linewidth}
			\centering
			\includegraphics[scale=0.36]{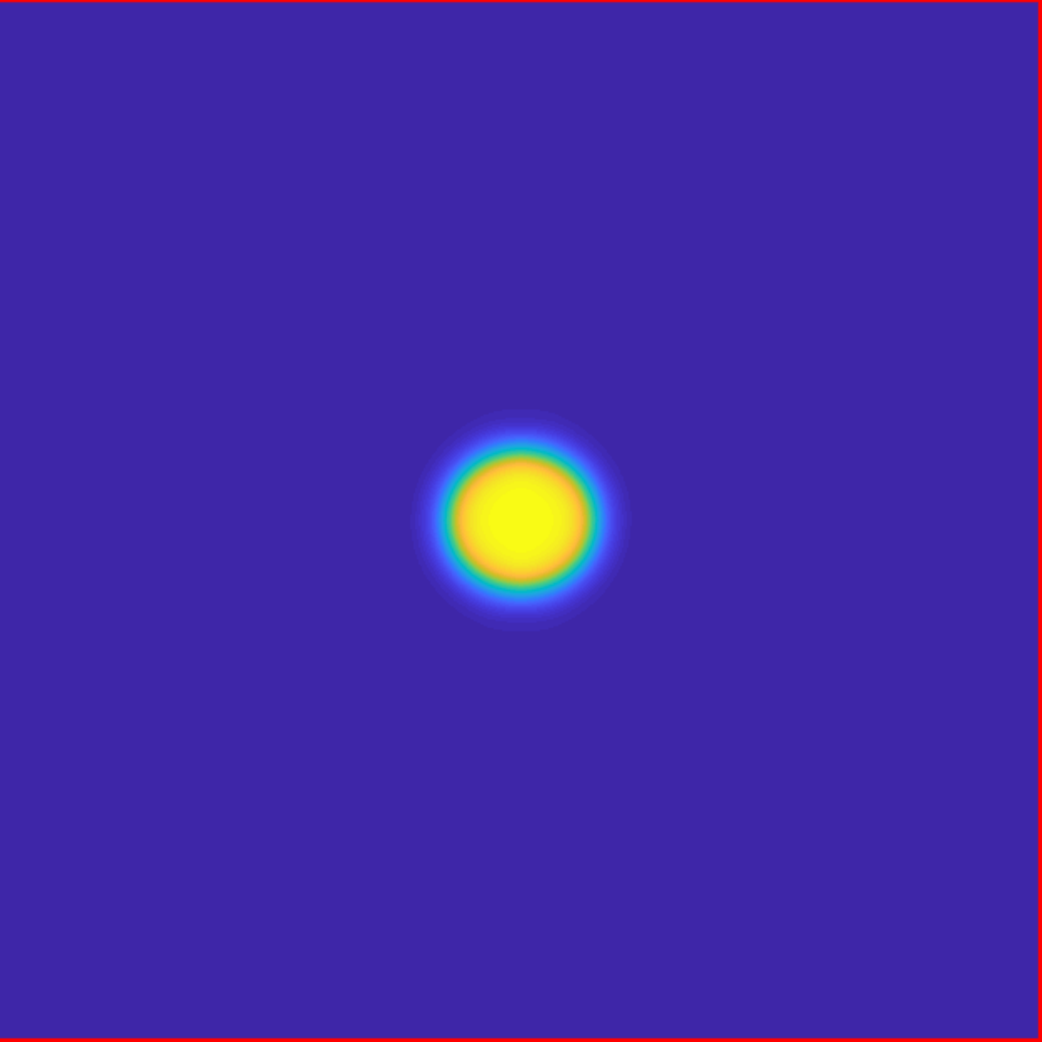}
		\end{minipage}
	}
	\caption{The process of merging and relaxing of two kissing circles obtained by the CNLFAC algorithm. 
	From left to right, top to bottom, $t=0$, $t=0.1$, $t=0.2$, $t=0.3$, $t=0.5$ and $t=0.6$, $\eta = 0.02$, $\lambda = 0.01$, $M = 10$, $\nu = 1$, $\Delta t = 0.005$.}
	\label{fig:CNLFACkiss-phi}
\end{figure}

\begin{figure}[tp]
	\centering
	\subfigure{
		\begin{minipage}[b]{.3\linewidth}
			\centering
			\includegraphics[scale=0.36]{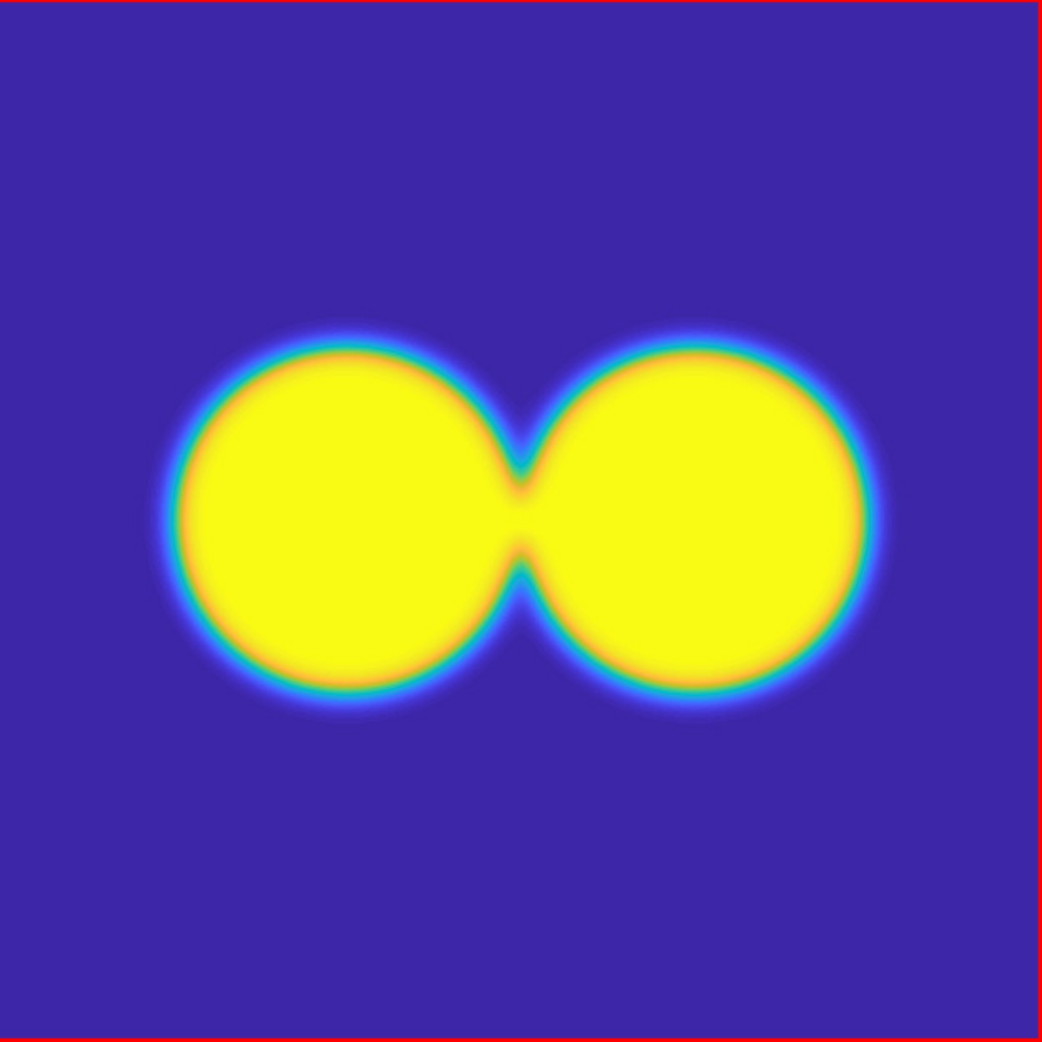}
		\end{minipage}
	}
	\subfigure{
		\begin{minipage}[b]{.3\linewidth}
			\centering
			\includegraphics[scale=0.36]{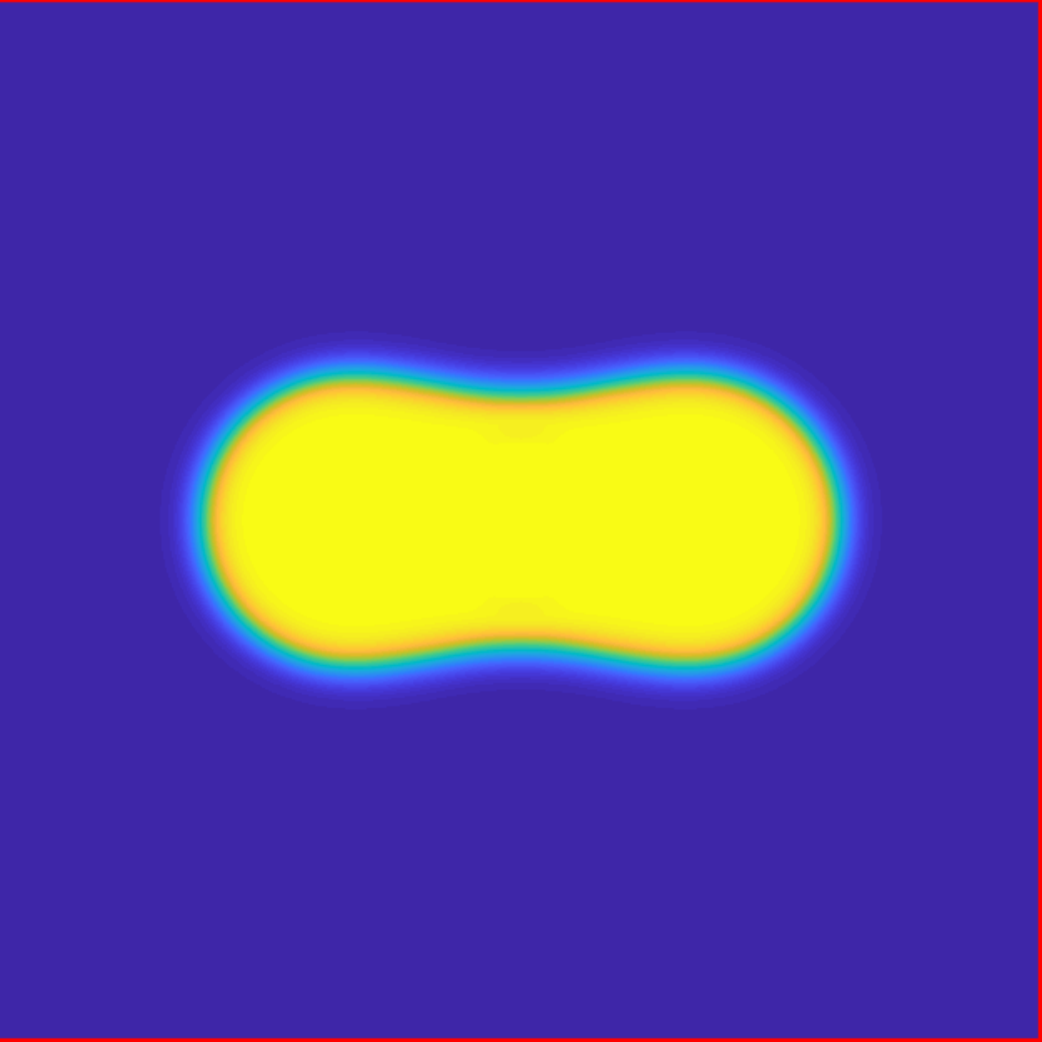}
		\end{minipage}
	}
	\subfigure{
		\begin{minipage}[b]{.3\linewidth}
			\centering
			\includegraphics[scale=0.36]{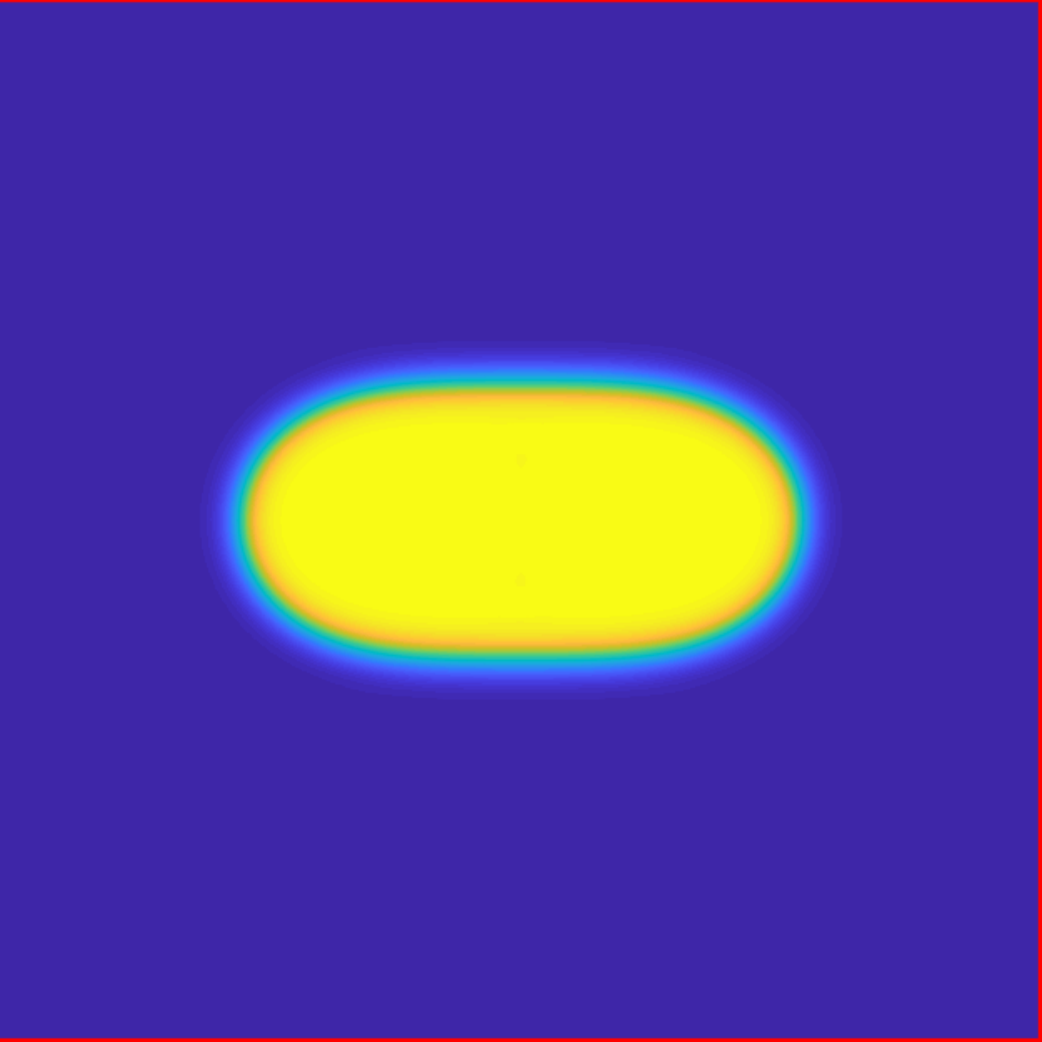}
		\end{minipage}
	}
	\subfigure{
		\begin{minipage}[b]{.3\linewidth}
			\centering
			\includegraphics[scale=0.36]{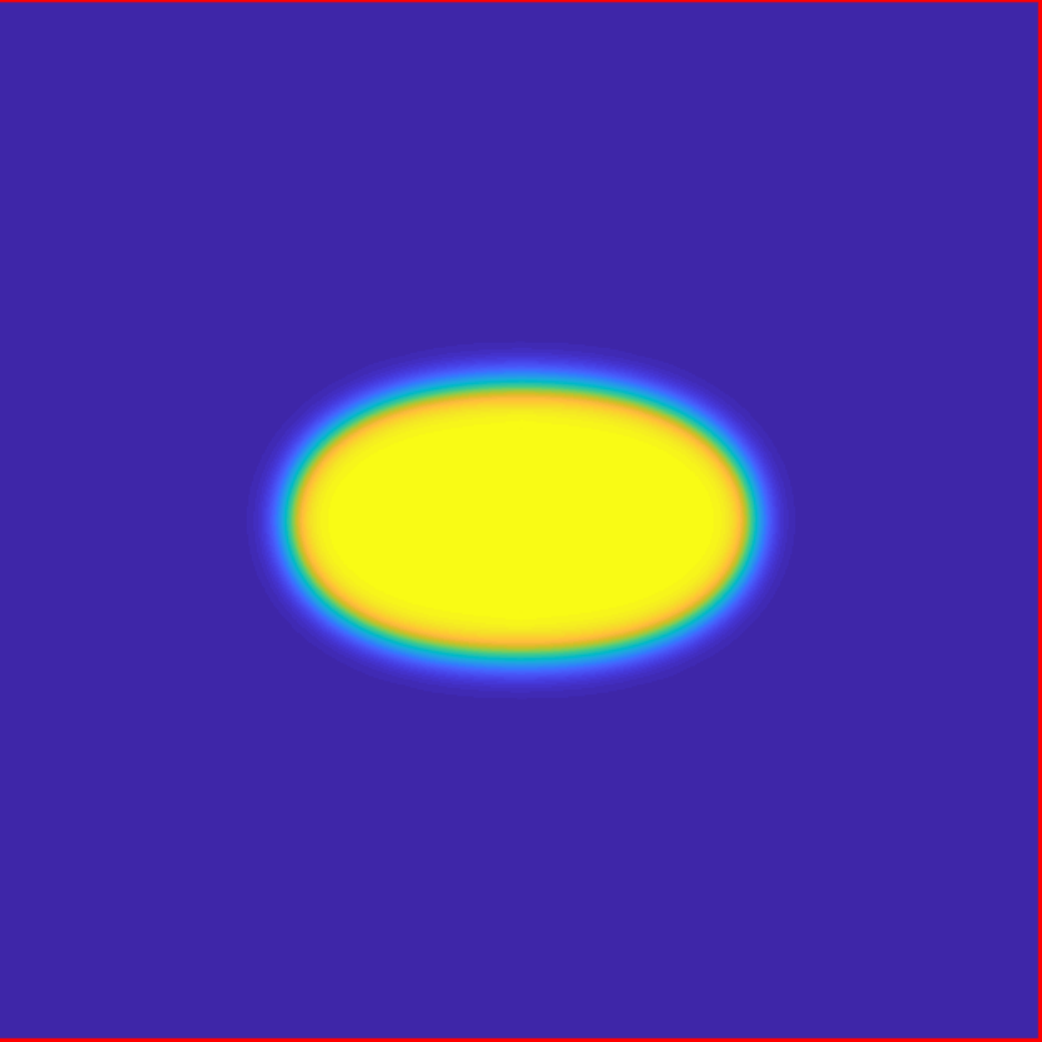}
		\end{minipage}
	}
	\subfigure{
		\begin{minipage}[b]{.3\linewidth}
			\centering
			\includegraphics[scale=0.36]{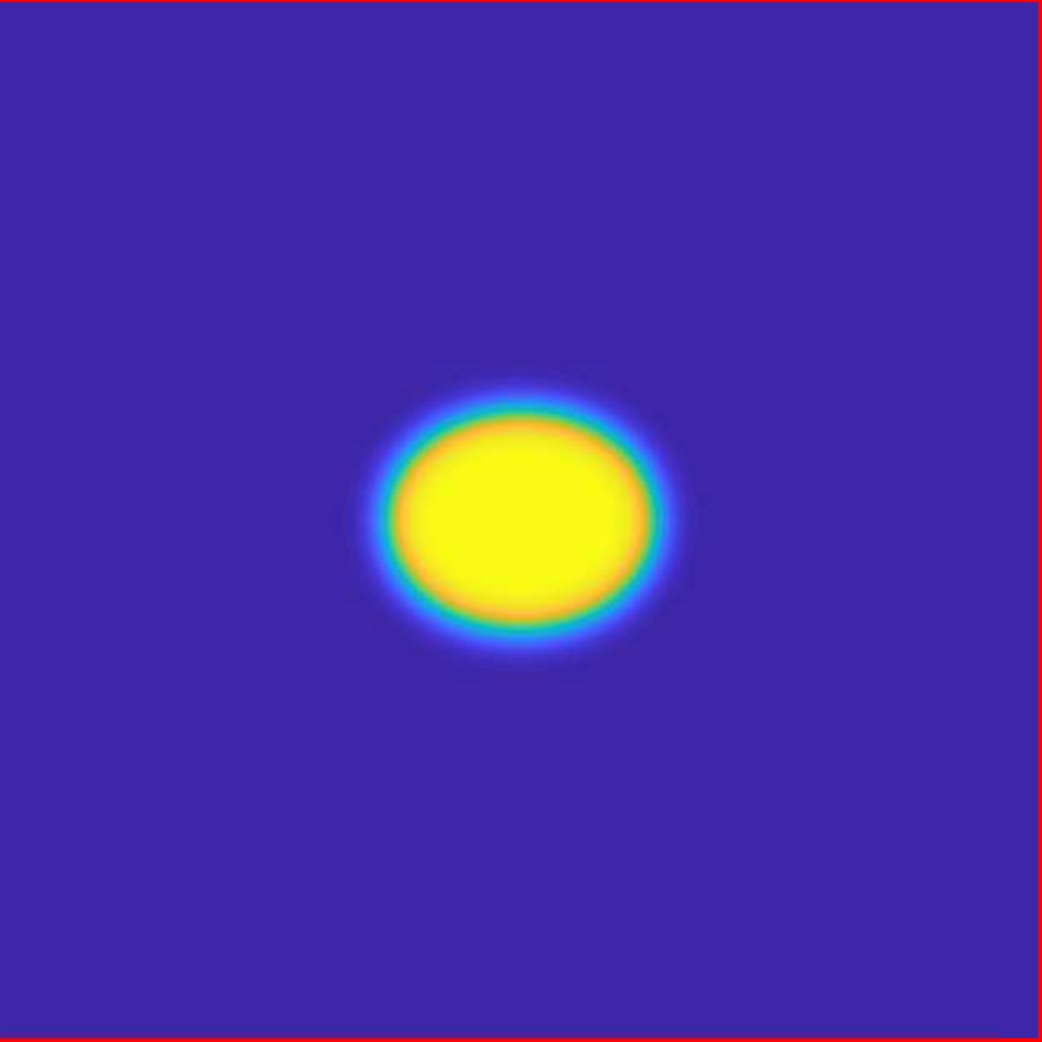}
		\end{minipage}
	}
	\subfigure{
		\begin{minipage}[b]{.3\linewidth}
			\centering
			\includegraphics[scale=0.36]{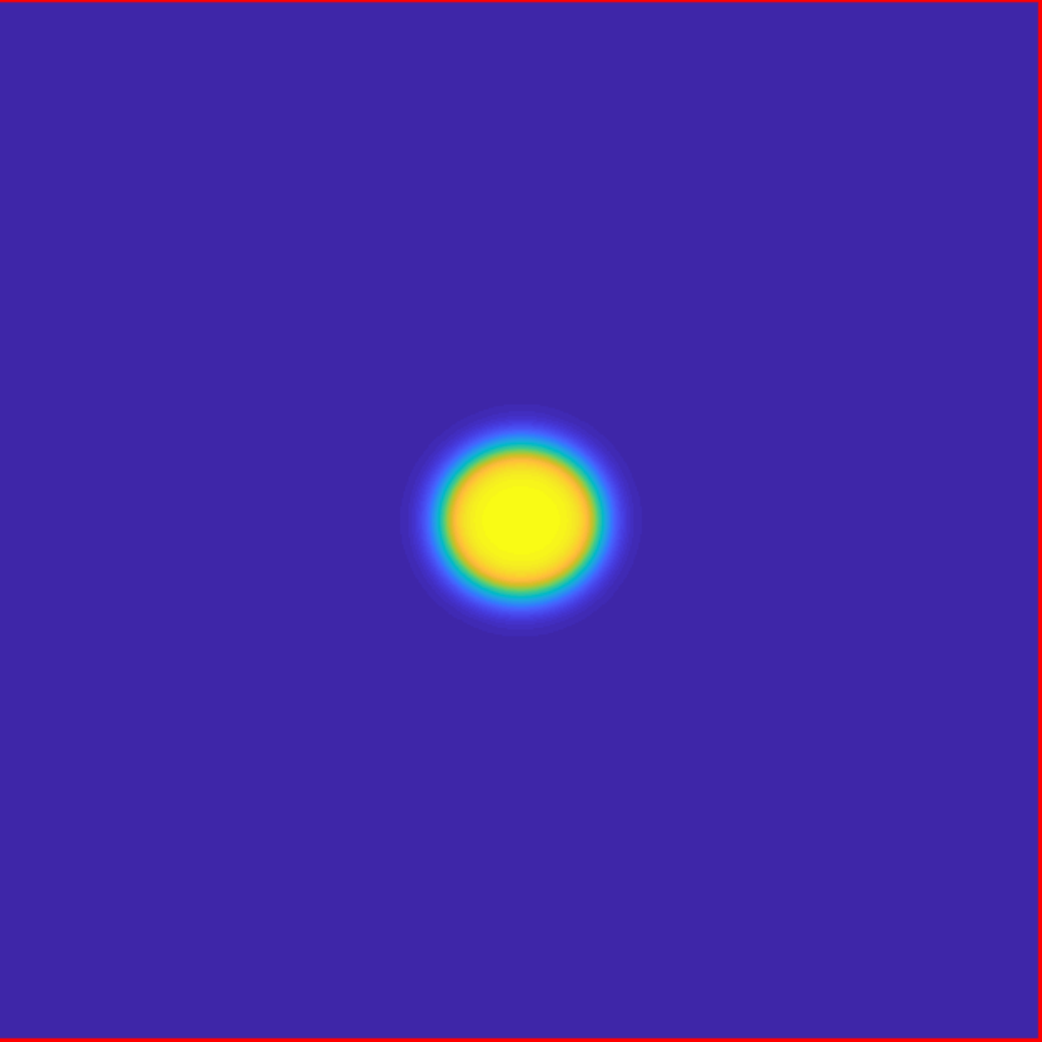}
		\end{minipage}
	}
	\caption{The process of merging and relaxing of two kissing circles obtained by the ACSAV algorithm. From left to right, top to bottom, $t=0$, $t=0.1$, $t=0.2$, $t=0.3$, $t=0.5$ and $t=0.6$, $\eta = 0.02$, $\lambda = 0.01$, $M = 10$, $\nu = 1$, $\Delta t = 0.005$.}
	\label{fig:ACSAVkiss-phi}
\end{figure}

\subsection{Computational efficiency}\label{65}
In this subsection, we present the CPU time for simulating spinodal decomposition in sec.~\ref{62} and shape relaxation in sec.~\ref{64}, by using the CNLFAC, ACSAV, and ACSAV-ECT algorithms.

The CPU time are reported in Table \ref{tab:cputime}. To be specific, the CPU time for simulating spinodal decomposition is counted with $T=5, \Delta t=0.01, h=2\pi/80$; and the CPU time for simulating shape relaxation is counted with $T=0.9, \Delta t=0.005, h=1.5/80$.  From Table \ref{tab:cputime} one can see that the ACSAV algorithm has a significant improvent on computational efficiency as compared to the CNLFAC scheme, as the CPU time can be reduced to 26\%. The ACSAV-ECT scheme features a further reduction on computational time, since the execution time is 7\% of that for CNLFAC method. Overall, the ACSAV-ECT scheme outperforms the other schemes since it is much faster while preserving similar accuracy.

\begin{table}[h!]
\caption{\noindent CPU times for simulating spinodal decomposition in sec.~\ref{62} and shape relaxation in sec.~\ref{64}.  The CPU time is also quantified in percentage relative to the CNLFAC scheme. }
\label{tab:cputime}
\begin{tabular}{|c||c|c||c|c|}
	\hline
	Scheme & CPU time for sec.~\ref{62} & PCT &  CPU time for sec.~\ref{64} & PCT \\
	\hline
	CNLFAC  & 7356.9s  &   100\% &  2701.0s  &   100\%  \\
	ACSAV  &  1912.8s  &  26.0\%  &  640.2s &   24\% \\
	ACSAV-ECT  &  515.5s  &   7.0\%  &  177.9s  &   6.6\%\\
	\hline
\end{tabular}
\end{table}

\section{Conclusion Remarks}\label{sec:sum}
In this paper, we have proposed three linear, unconditionally long time stable, second order, decoupling methods for solving the ACNS system. The first method, namely the CNLFAC scheme, is based on the Crank--Nicolson leap-frog time discretization, the Lagrange multiplier method for linearizing the Allen--Cahn equations, and an artificial compression technique for decoupling the velocity and pressure in the NS equations. 
The second scheme, namely ACSAV, is formulated by incorporating an SAV decoupling strategy into the CNLFAC method so that the Allen--Cahn and Navier--Stokes equations are numerically decoupled and a highly efficient, fully decoupled scheme is built. The third is another version of ACSAV with explicit convection term, i.e.~ACSAV-ECT.
We prove that all three schemes are unconditionally stable without any time step conditions.
Numerical experiments are performed to verify that our schemes are of second order accuracy in time and unconditionally stable. Efficiency tests show that the ACSAV and ACSAV-ECT schemes can significantly reduce the execution time as compared to the CNLFAC scheme.

It is worth noting that there are very few second order, unconditionally stable, fully decoupled numerical schemes for solving hydrodynamics coupled phase field models. 
The idea here (CNLF+AC+SAV) serves as a template for designing unconditionally stable and fully decoupled schemes for solving other related phase field models,  such as the Cahn--Hilliard--Navier--Stokes equations, the Cahn--Hilliard--Hele--Shaw system, and phase field fluid models of variable densities. 

\section*{Acknowledgments}
 R. Cao and H. Yang were supported in part by the National Natural Science Foundation of China under grant 11801348, the key research projects of general universities in Guangdong Province (grant no.~2019KZDXM034), and the basic research and applied basic research projects in Guangdong Province (Projects of Guangdong, Hong Kong and Macao Center for Applied Mathematics, grant no.~2020B1515310018).
 N. Jiang was partially supported by the US National Science Foundation grants DMS-1720001 and DMS-2120413.

\end{document}